\documentclass[journal]{IEEEtran}

\usepackage[english]{babel}
\usepackage[utf8x]{inputenc}
\usepackage{amsmath}
\usepackage{graphicx}
\usepackage[colorinlistoftodos]{todonotes}
\usepackage{url}
\usepackage{algorithmic}
\usepackage{array}

\RequirePackage[colorlinks = true,
            linkcolor = black,
            urlcolor  = black,
            citecolor = black,
            anchorcolor = black]{hyperref}

\usepackage{mathtools, amssymb, amsthm, enumerate, graphicx, color, soul, enumitem}

\newcommand{\A}{\mathcal{A}}

\newcommand{\C}{\mathbb{C}}

\newcommand{\E}{\mathbb{E}}

\renewcommand{\H}{\mathcal{H}}

\renewcommand{\L}{\mathcal{L}}

\newcommand{\R}{\mathbb{R}}

\newcommand{\V}{\mathbb{V}}
\newcommand{\X}{\mathcal{X}}

\newcommand{\Z}{\mathbb{Z}}

\newcommand{\h}{\mathcal{H}}

\newcommand{\m}{\mathcal{M}}
\renewcommand{\r}{\mathcal{R}}
\newcommand{\x}{\mathcal{X}}

\newcommand{\Pt}{P_\theta}
\newcommand{\pt}{p_\theta}
\renewcommand{\th}{\hat{\theta}}
\newcommand{\Pth}{P_{\th}}
\newcommand{\pth}{p_{\th}}

\newcommand{\Te}{\Theta_\epsilon}
\newcommand{\te}{\theta_\epsilon}

\newcommand{\ts}{\theta^*}

\renewcommand{\tt}{\tilde{\theta}}

\newcommand{\Pts}{P_{\ts}}

\newcommand{\tr}{\text{tr} \,}

\newcommand{\iid}{\overset{iid}{\sim}}

\newtheorem{theorem}{Theorem}[section]
\newtheorem{lemma}[theorem]{Lemma}

\newtheorem{corollary}[theorem]{Corollary}

\begin{document}

\title{Finite-sample risk bounds for maximum likelihood estimation with arbitrary penalties}
\author{W. D. Brinda and
    Jason~M.~Klusowski,~\IEEEmembership{Student Member,~IEEE}
\thanks{W. D. Brinda and Jason M. Klusowski are with the Department of Statistics \& Data Science, Yale University, New Haven, CT, USA e-mail: \{william.brinda, jason.klusowski\}@yale.edu.}
\thanks{}}

\maketitle

\begin{abstract}
The MDL two-part coding \emph{index of resolvability} provides a finite-sample upper bound on the statistical risk of penalized likelihood estimators over countable models. However, the bound does not apply to unpenalized maximum likelihood estimation or procedures with exceedingly small penalties. In this paper, we point out a more general inequality that holds for arbitrary penalties. In addition, this approach makes it possible to derive exact risk bounds of order $1/n$ for iid parametric models, which improves on the order $(\log n)/n$ resolvability bounds. We conclude by discussing implications for adaptive estimation.
\end{abstract}


\IEEEpeerreviewmaketitle

\section{Introduction}

\IEEEPARstart{A} remarkably general method for bounding the statistical risk of penalized likelihood estimators comes from work on two-part coding, one of the minimum description length (MDL) approaches to statistical inference. Two-part coding MDL prescribes assigning codelengths to a model (or model class) then selecting the distribution that provides the most efficient description of one's data \cite{rissanen1978}. The total description length has two parts: the part that specifies a distribution within the model (as well as a model within the model class if necessary) and the part that specifies the data with reference to the specified distribution. If the codelengths are exactly Kraft-valid, this approach is equivalent to Bayesian maximum a posteriori (MAP) estimation, in that the two parts correspond to log reciprocal of prior and log reciprocal of likelihood respectively. More generally, one can call the part of the codelength specifying the distribution a \emph{penalty} term; it is called the \emph{complexity} in MDL literature.

Let $(\Theta, \L)$ denote a discrete set indexing distributions along with a complexity function. With $X \sim P$, the (pointwise) \emph{redundancy} of any $\theta \in \Theta$ is its two-part codelength minus $\log (1/p(X))$, the codelength one gets by using $P$ as the coding distribution.\footnote{For now, we mean that $P$ governs the entirety of the data. The notion of sample size and iid assumptions are not essential to the bounds, as will be seen in the statement of Theorem~\ref{thm:arbitrary-penalty}. Specialization to iid data will be discussed thereafter.} The expectation of redundancy is the relative entropy from $P$ to $\Pt$ plus $\L(\theta)$. Let $\ts \in \Theta$ denote the minimizer of expected redundancy; it is the average-case optimal representative from $(\Theta, \L)$ when the true distribution is $P$. Its expected redundancy will be denoted
\begin{align*}
\r_{\Theta, \L}(P) := \inf_{\theta \in \Theta} \{ D(P \| \Pt) + \L(\theta) \},
\end{align*}
or in the context of iid data $X^n \sim P^n$ and iid modeling $\{\Pt^n : \theta \in \Theta\}$, its expected redundancy rate is denoted
\begin{align*}
\r_{\Theta, \L}^{(n)}(P) := \inf_{\theta \in \Theta} \left\{ D(P \| \Pt) + \frac{\L(\theta)}{n} \right\}.
\end{align*}

Interestingly, \cite{barron1991} showed that if the complexity function is large enough, then the corresponding penalized likelihood estimator outperforms the best-case average representative. Specifically, the statistical risk is bounded by $\r_{\Theta, \L}(P)$; that result is stated for iid sampling in (\ref{eq:usual}) below.\footnote{Throughout the paper, we will refer to this inequality as ``the resolvability bound," but realize that there are a variety of related resolvability bounds in other contexts. They involve comparing risk to a codelength and lead to bounds that are suboptimal by a $\log n$ factor.}

There are a number of attractive features of the resolvability bound; we will highlight four. One of the most powerful aspects of the resolvability bound is the ease with which it can be used to devise adaptive estimation procedures for which the bound applies. For instance, to use a class of nested models rather than a single model, one only needs to tack on an additional penalty term corresponding to a codelength used to specify the selected model within the class.

Another nice feature is its generality: the inequality statement only requires that the data-generating distribution has finite relative entropy to some probability measure in the model.\footnote{Although the forthcoming resolvability bounds (i.e., as in \eqref{eq:usual} with $ \mathcal{L} $ that is at least twice a codelength function) are valid under misspecification, they do \emph{not} in general imply consistency in the sense that the corresponding penalized estimator eventually converges to the element $ \theta' $ of $\Theta$ that minimizes KL or Hellinger to the truth $ P $. Indeed, there are various examples \cite{grunwald2007b} in which the twice-codelength penalized estimator is inconsistent (i.e., provably never converges to $ \theta' $).} In practice, the common assumptions of other risk bound methods, for instance, that the generating distribution belongs to the model, are unlikely to be exactly true.

A third valuable property of the bound is its exactness for finite samples. Many risk bound methods only provide asymptotic bounds. But such results do not imply anything exact for a data analyst with a specific sample.

Lastly, the resolvability bound uses a meaningful loss function: $\alpha$-Renyi divergence \cite{renyi1961} with $\alpha \in (0, 1)$. For convenience, we specialize our discussion and our present work to Bhattacharyya divergence \cite{bhattacharyya1943} which is the $\tfrac{1}{2}$-Renyi divergence.
\begin{align*}
D_B(P, Q) := 2 \log \frac{1}{A(P, Q)},
\end{align*}
where $A$ denotes the Hellinger affinity
\begin{align*}
A(P, Q) &:= \int \sqrt{p(x) q(x)} dx\\
 &= \E_{X \sim P} \sqrt{\frac{q(X)}{p(X)}}.
\end{align*}
Like relative entropy, $D_B$ decomposes product measures into sums; that is,
\begin{align*}
A(P^n, Q^n) = A(P, Q)^n \quad \text{thus} \quad D_B(P^n, Q^n) = n D_B(P, Q).
\end{align*}

Bhattacharyya divergence is bounded below by squared Hellinger distance (using $\log 1/x \geq 1-x$) and above by relative entropy (using Jensen's inequality). Importantly, it has a strictly increasing relationship with squared Hellinger distance $D_H$, which is an $f$-divergence:
\begin{align*}
D_B = 2 \log \frac{1}{1 - D_H/2}.
\end{align*}
As such, it inherits desirable $f$-divergence properties such as the data processing inequality. Also, it is clear from the definition that $D_B$ is parametrization-invariant. For many more properties of $D_B$, including its bound on total variation distance, see \cite{erven2014}.

Next, we make note of some of the limitations of the resolvability bound. One complaint is that it is for discrete parameter sets, while people generally want to optimize penalized likelihood over a continuous parameter space. In practice, one typically selects a parameter value that is rounded to a fixed precision, so in effect the selection is from a discretized space. However, for mathematical convenience, it is nice to have risk bounds for the theoretical optimizer. A method to extend the resolvability bound to continuous models was introduced by \cite{barron2008}; in that paper, the method was specialized to estimation of a log density by linear combinations from a finite dictionary with an $l_1$ penalty on the coefficients. More recently, \cite{chatterjee2014b} worked out the continuous extension for Gaussian graphical models (building on \cite{luo2009}) with $l_1$ penalty assuming the model is well-specified and for linear regression with $l_0$ penalty assuming the true error distribution is Gaussian. These results are explained in more detail by \cite{chatterjee2014}, where the extension for the $l_1$ penalty for linear regression is also shown, again assuming the true error distribution is Gaussian.

Another limitation is that the resolvability bound needs a large enough penalty; it must have a finite Kraft sum. This paper provides a more general inequality that escapes such a requirement and therefore applies even to unpenalized maximum likelihood estimation. The resulting bound retains the four desirable properties we highlighted above, but loses the coding and resolvability interpretations.

Finally, the resolvability bounds for smooth parametric iid modeling are of order $(\log n)/n$ and cannot be improved, according to \cite{rissanen1986}, whereas under regularity conditions (for which Bhattacharyya divergence is locally equivalent to one-half relative entropy, according to \cite{barron2008}) the optimal Bhattacharyya risk is of order $1/n$ \cite{barron1998}. Our variant on the resolvability method leads to the possibility of deriving exact bounds of order $1/n$.

Our bounds can be used for the penalized MLE over a discretization of an \emph{unbounded} parameter space under a power decay condition on the Hellinger affinity, as in Theorems \ref{thm:gaussian-bound-center-power-bound-tail} and \ref{thm:gaussian-power-any-sample-size}. We show that such a condition is satisfied by exponential families of distributions with a boundedness assumption on the largest eigenvalue of the covariance matrix of their sufficient statistics (see Lemma \ref{lem:exponential-affinity-guassian-bound}). For these models and others, we establish order $ 1/n $ bounds for the Bhattacharyya risk. The primary focus of this paper is to develop new tools towards this end.

One highly relevant line of work is \cite{zhang2006}, where he established a more general resolvability risk bound for ``posterior'' distributions on the parameter space. Implications for penalized MLEs come from forcing the ``posteriors" to be point-masses. He derives risk bounds that have the form of $\r_{\Theta, \L}^{(n)}(P)$ plus a ``corrective'' term, which is comparable to the form of our results. Indeed, as we will point out, one of our corollaries nearly coincides with \cite[Thm 4.2]{zhang2006} but works with arbitrary penalties.

The trick we employ is to introduce an arbitrary function $L$, which we call a \emph{pseudo-penalty}, that adds to the penalty $\L$; strategic choices of pseudo-penalty can help to control the ``penalty summation" over the model. The resulting risk bound has an additional $\E L(\th)$ term that must be dealt with.

In Section~\ref{sec:discrete}, we prove our more general version of the resolvability bound inequality using a derivation closely analogous to the one by \cite{li1999}. We then explore corollaries that arise from various choices of pseudo-penalty. In Section~\ref{sec:adaptive-modeling}, we explain how our approach applies in the context of adaptive modeling. Additional work can be found in \cite{brinda2018}, including some simple concrete examples \cite[``Simples concrete examples", Sec 2.1.2]{brinda2018}, extension to continuous models \cite[``Continuous parameter spaces", Sec 2.2]{brinda2018}, and an application to Gaussian mixtures \cite[Chap 4]{brinda2018}.

Every result labeled a Theorem or Lemma has a formal proof, some of which are in the Appendix. Any result labeled a Corollary is an immediate consequence of previously stated results and thus no formal proof is provided. For any random vector $X$, the notation $\C X$ means the covariance matrix, while $\V X$ represents its trace $\E \| X - \E X \|^2$. The notation $\lambda_j(\cdot)$ means the $j$th eigenvalue of the matrix argument. Whenever a capital letter has been introduced to represent a probability distribution, the corresponding lower-case letter will represent a density for the measure with respect to either Lebesgue or counting measure. The \emph{penalized MLE} is the (random) parameter that maximizes log-likelihood minus penalty. The notation $D(P \| \Theta)$ represents the infimum relative entropy from $P$ to distributions indexed by the model $\Theta$. Multiplication and division take precedence over $\wedge$ and $\vee$; for instance, $ab \wedge c$ means $(ab) \wedge c$.

\section{Models with countable cardinality}\label{sec:discrete}

Let us begin with countable (e.g. discretized) models, which were the original context for the MDL penalized likelihood risk bounds. We will show that a generalization of that technique works for arbitrary penalties. The only assumption we need is that for any possible data, there exists a (not necessarily unique) minimizer of penalized likelihood.\footnote{We will say ``the" penalized MLE, even though we do not require uniqueness; any scheme can be used for breaking ties.} This existence requirement will be implicit throughout our paper. Theorem~\ref{thm:arbitrary-penalty} gives a general result that is agnostic about any structure within the data; the consequence for iid data with sample size $n$ is pointed out after the proof.

\begin{theorem}\label{thm:arbitrary-penalty}
Let $X \sim P$, and let $\th$ be the penalized MLE over $\Theta$ indexing a countable model with penalty $\L$. Then for any $L: \Theta \rightarrow \R$,
\begin{align*}
\E D_B(P, \Pth) & \leq \r_{\Theta, \L}(P) + \\ & \qquad 2 \log \sum_{\theta \in \Theta} e^{-\tfrac{1}{2}[\L(\theta) + L(\theta)]} + \E L(\th).
\end{align*}
\end{theorem}

\begin{proof}
We follow the pattern of Jonathan Li's version of the resolvability bound proof \cite{li1999}.
\begin{align*}
D_B(P, \Pth) &:= 2 \log \frac{1}{A(P, \Pth)}\\
 &= 2 \log \frac{\sqrt{\pth(X)/p(X)} e^{-\tfrac{1}{2}[\L(\th) + L(\th)]}}{A(P, \Pth)} + \\ & \qquad \log \frac{p(X)}{\pth(X)} + \L(\th) + L(\th)\\
 &\leq 2 \log \sum_{\theta \in \Theta} \frac{\sqrt{\pt(X)/p(X)} e^{-\tfrac{1}{2}[\L(\theta) + L(\theta)]}}{A(P, \Pt)} + \\ & \qquad \log \frac{p(X)}{\pth(X)} + \L(\th) + L(\th).
\end{align*}
We were able to bound the random quantity by the sum over all $\theta \in \Theta$ because each of these terms is non-negative.

We will take the expectation of both sides for $X \sim P$. To deal with the first term, we use Jensen's inequality and the definition of Hellinger affinity.
\begin{align*}
& 2\E \log \sum_{\theta \in \Theta} \frac{\sqrt{\pt(X)/p(X)} e^{-\tfrac{1}{2}[\L(\theta) + L(\theta)]}}{A(P, \Pt)} \\ & \leq 2 \log \sum_{\theta \in \Theta} \E \frac{\sqrt{\pt(X)/p(X)} e^{-\tfrac{1}{2}[\L(\theta) + L(\theta)]}}{A(P, \Pt)}
 \\ & = 2 \log \sum_{\theta \in \Theta} e^{-\tfrac{1}{2}[\L(\theta) + L(\theta)]}.
\end{align*}

Returning to the overall inequality, we have
\begin{align*}
\E D_B(P, \Pth) &\leq 2 \log \sum_{\theta \in \Theta} e^{-\tfrac{1}{2}[\L(\theta) + L(\theta)]} + \\ & \qquad \E \left[\log \frac{p(X)}{\pth(X)} + \L(\th)\right] + \E L(\th)\\
 &= 2 \log \sum_{\theta \in \Theta} e^{-\tfrac{1}{2}[\L(\theta) + L(\theta)]} + \\ & \qquad \E \min_{\theta \in \Theta} \left\{ \log \frac{p(X)}{\pt(X)} + \L(\theta) \right\} + \E L(\th)\\
 &\leq 2 \log \sum_{\theta \in \Theta} e^{-\tfrac{1}{2}[\L(\theta) + L(\theta)]} + \\ & \qquad \inf_{\theta \in \Theta} \E \left\{ \log \frac{p(X)}{\pt(X)} + \L(\theta) \right\} + \E L(\th)\\
 &= 2 \log \sum_{\theta \in \Theta} e^{-\tfrac{1}{2}[\L(\theta) + L(\theta)]} + \\ & \qquad \inf_{\theta \in \Theta} \left\{ D(P \| \Pt) + \L(\theta) \right\} + \E L(\th).
\end{align*}
\end{proof}

Suppose now that the data comprise $n$ iid observations and are modeled as such; in other words, the data has the form $X^n \sim P^n$, and the model has the form $\{ \Pt^n : \theta \in \Theta \}$. Because $D_B(P^n, \Pth^n) = n D_B(P, \Pth)$ and $D(P^n \| \Pt^n) = n D(P \| \Pt)$, we can divide both sides of Theorem~\ref{thm:arbitrary-penalty} by $n$ to reveal the role of sample size in this context:
\begin{align*}
\E D_B(P, \Pth) & \leq \r_{\Theta, \L}^{(n)}(P) + \\ & \qquad \frac{2 \log \sum_{\theta \in \Theta} e^{-\tfrac{1}{2}[\L(\theta) + L(\theta)]} + \E L(\th)}{n}.
\end{align*}

We will see three major advantages to Theorem~\ref{thm:arbitrary-penalty}. The most obvious is that it can handle cases in which the sum of exponential negative half penalties is infinite; unpenalized estimation, for example, has $\L$ identically zero. One consequence of this is that the resolvability method for minimax risk upper bounds can be extended to models that are not finitely covered by relative entropy balls. We will also find that Theorem~\ref{thm:arbitrary-penalty} enables us to derive exact risk bounds of order $1/n$ rather than the usual $(\log n)/n$ resolvability bounds.

In many cases, it is convenient to have only the $L$ function in the summation. Substituting $L - \L$ as the pseudo-penalty in Theorem~\ref{thm:arbitrary-penalty} gives us a corollary that moves $\L$ out of the summation.
\begin{corollary}\label{cor:subtract-penalty}
Let $X \sim P$, and let $\th$ be the penalized MLE over $\Theta$ indexing a countable model with penalty $\L$. Then for any $L: \Theta \rightarrow \R$,
\begin{align*}
\E D_B(P, \Pth) & \leq \r_{\Theta, \L}(P) + \\ & \qquad 2 \log \sum_{\theta \in \Theta} e^{-\tfrac{1}{2}L(\theta)} + \E L(\th) - \E \L(\th).
\end{align*}
\end{corollary}
The iid data and model version is
\begin{align*}
\E D_B(P, \Pth) & \leq \r_{\Theta, \L}^{(n)}(P) + \\ & \qquad \frac{2 \log \sum_{\theta \in \Theta} e^{-\tfrac{1}{2}L(\theta)} + \E L(\th) - \E \L(\th)}{n}.
\end{align*}
We will use the term \emph{pseudo-penalty} for the function labeled $L$ in \emph{either} Theorem~\ref{thm:arbitrary-penalty} or Corollary~\ref{cor:subtract-penalty}. Note that $L$ is allowed to depend on $P$ but not on the data.

A probabilistic loss bound can also be derived for the difference between the loss and the redundancy plus pseudo-penalty.

\begin{theorem}\label{thm:probabilistic-loss-bound}
Let $X \sim P$, and let $\th$ be the penalized MLE over $\Theta$ indexing a countable model with penalty $\L$. Then for any $L: \Theta \rightarrow \R$,
\begin{align*}
& P \left\{ D_B(P, \Pth) - \left[ \log \frac{p(X)}{\pth(X)} + \L(\th) + L(\th)\right] \geq t \right\} \\ & \leq e^{-t/2} \sum_{\theta \in \Theta} e^{-\tfrac{1}{2}[\L(\theta) + L(\theta)]}.
\end{align*}
\end{theorem}

\begin{proof}
Following the steps described in \cite[Theorem 2.3]{barron2008}, we use Markov's inequality then bound a non-negative random variable by the sum of its possible values.
\begin{align*}
& P \left\{ D_B(P, \Pth) - \left[ \log \frac{p(X)}{\pth(X)} + \L(\th) + L(\th)\right] \geq t \right\} \\ & = P \left\{ 2 \log \frac{\sqrt{\pth(X)/p(X)}e^{-\frac{1}{2}[\L(\th) + L(\th)]}}{A(P, \Pth)} \geq t \right\}\\
 &= P \left\{ \frac{\sqrt{\pth(X)/p(X)}e^{-\frac{1}{2}[\L(\th) + L(\th)]}}{A(P, \Pth)} \geq e^{t/2} \right\}\\
 &\leq e^{-t/2} \E \frac{\sqrt{\pth(X)/p(X)}e^{-\frac{1}{2}[\L(\th) + L(\th)]}}{A(P, \Pth)}\\
 &\leq e^{-t/2} \sum_{\theta \in \Theta} e^{-\tfrac{1}{2}[\L(\theta) + L(\theta)]}.
\end{align*}
\end{proof}

For iid data $X^n \iid P$ and an iid model, Theorem~\ref{thm:probabilistic-loss-bound} implies
\begin{align*}
& P \left\{ D_B(P, \Pth) - \frac{1}{n}\left[ \sum_i \log \frac{p(X_i)}{\pth(X_i)} + \L(\th) + L(\th)\right] \geq t \right\} \\ &\leq e^{-nt/2} \sum_{\theta \in \Theta} e^{-\tfrac{1}{2}[\L(\theta) + L(\theta)]}.
\end{align*}

Several of our corollaries have $\L$ and $L$ designed to make $\sum_{\theta \in \Theta} e^{-\tfrac{1}{2}[\L(\theta) + L(\theta)]} \leq 1$. In such cases, the difference between loss and the point-wise redundancy plus pseudo-penalty is stochastically less than an exponential random variable.

Often the countable model of interest is a discretization of a continuous model. Given any $\epsilon > 0$, an \emph{$\epsilon$-discretization} of $\R^d$ is $v + \epsilon \Z^d$, by which we mean $\{ v + m \epsilon : m \in \Z^d \}$ for some $v \in \R^d$. An \emph{$\epsilon$-discretization} of $\Theta \subseteq \R^d$ is a set of the form $\Theta \cap (v + \epsilon \Z^d)$. See Section~\ref{sec:behavior-r} for a discussion of the behavior of $\r_{\Theta, \L}^{(n)}(P)$ in that context.

To derive useful consequences of the above results, we will explore some convenient choices of pseudo-penalty: zero, Bhattacharyya divergence, log reciprocal pmf of $\th$, quadratic forms, and the penalty. We specialize to the iid data and model setting for the remainder of this document to highlight the fact that many of the exact risk bounds we derive are of order $1/n$ in that case.

\subsection{Zero as pseudo-penalty}

Setting $L$ to zero gives us the traditional resolvability bound, which we review in this section.
\begin{corollary}\label{cor:set-l-zero}
Assume $X^n \iid P$, and let $\th$ be the penalized MLE over $\Theta$ indexing a countable iid model with penalty $\L$. Then
\begin{align*}
\E D_B(P, \Pth) \leq \r_{\Theta, \L}^{(n)}(P) + \frac{2 \log \sum_{\theta \in \Theta} e^{-\tfrac{1}{2}\L(\theta)}}{n}.
\end{align*}
\end{corollary}

The usual statement of the resolvability bound \cite{barron2008} assumes $\L$ is at least twice a codelength function, so that it is large enough for the sum of exponential terms to be no greater than $1$. That is,
\begin{align}\label{eq:twice-kraft}
\sum_{\theta \in \Theta} e^{-\tfrac{1}{2}\L(\theta)} \leq 1
\end{align}
implies
\begin{align}\label{eq:usual}
\E D_B(P, \Pth) \leq \r_{\Theta, \L}^{(n)}(P).
\end{align}
The quantity on the right-hand side of (\ref{eq:usual}) is called the \emph{index of resolvability} of $(\Theta, \L)$ for $P$ at sample size $n$. Any corresponding minimizer $\ts \in \Theta$ is considered to index an average-case optimal representative for $P$ at sample size $n$.

In fact, for any finite sum $z := \sum_{\theta \in \Theta} e^{-\tfrac{1}{2}\L(\theta)}$, the maximizer of the penalized likelihood is also the maximizer with penalty $\tilde{\L} := \L + 2 \log z$. Thus one has a resolvability bound of the form (\ref{eq:usual}) with the equivalent penalty $\tilde{\L}$, which satisfies (\ref{eq:twice-kraft}) with equality.

Additionally, the resolvability bounds give an exact upper bound on the minimax risk for any model $\Theta$ that can be covered by finitely many relative entropy balls of radius $\epsilon^2$; the log of the minimal covering number is called the \emph{KL-metric entropy} $\m(\epsilon)$. These balls' center points are called a \emph{KL-net}; we will denote the net by $\Te$. With data $X^n \iid \Pts$ for any $\ts \in \Theta$, the MLE restricted to $\Te$ has the resolvability risk bound
\begin{align*}
\E D_B(\Pts, \Pth) &\leq \inf_{\theta \in \Te} \left\{ D(\Pts \| \Pt) + \frac{2 \m(\epsilon)}{n} \right\}\\
 &= \inf_{\theta \in \Te} D(\Pts \| \Pt) + \frac{2 \m(\epsilon)}{n}\\
 &\leq \epsilon^2 + \frac{2 \m(\epsilon)}{n}.
\end{align*}
If an explicit bound for $\m(\epsilon)$ is known, then the overall risk bound can be optimized over the radius $\epsilon$ --- see for instance \cite[Section 1.5]{barron2008}.

Because this approach to upper bounding minimax risk requires twice-Kraft-valid codelengths, it only applies to models that can be covered by finitely many relative entropy balls. However, Corollary~\ref{cor:subtract-penalty} reveals new possibilities for establishing minimax upper bounds even if the cover is infinite. Given any $L$, one can use any constant penalty that is at least as large as $2 \log \sum e^{-\tfrac{1}{2}L(\theta)} + \E L(\th)$ where $\th$ is the unpenalized MLE on the net and the summation is taken over those points.\footnote{Putting $\L = 0$ into either Theorem~\ref{thm:arbitrary-penalty} or Corollary~\ref{cor:subtract-penalty} would give us the same idea.} For a minimax result, one still needs this quantity to be uniformly bounded over all data-generating distribution $\ts \in \Theta$. See Corollary~\ref{cor:exponential-affinity-guassian-minimax} below as an example.

\subsection{Bhattacharyya divergence as pseudo-penalty}\label{sec:bhattacharyya-pseudo-penalty}

Important corollaries\footnote{Our Corollary~\ref{cor:bhattacharyya-pseudo-penalty} was inspired by the very closely related result of \cite[Thm 4.2]{zhang2006}.} to Theorems~\ref{thm:arbitrary-penalty} and~\ref{cor:subtract-penalty} come from setting the pseudo-penalty equal to $\alpha D_B(P, \Pt)$; the expected pseudo-penalty is proportional to the risk, so that term can be subtracted from both sides. For the iid scenario, we also use the product property of Hellinger affinity: $A(P^n, \Pt^n) = A(P, \Pt)^n$.

The following corollaries serve as the starting point for the main bounds in Theorems \ref{thm:gaussian-power-any-sample-size} and \ref{thm:gaussian-bound-center-power-bound-tail}, after which, more refined techniques are used in controlling the two terms in \eqref{eq:B1} and \eqref{eq:B2}.
\begin{corollary}\label{cor:bhattacharyya-pseudo-penalty}
Assume $X^n \iid P$, and let $\th$ be the penalized MLE over $\Theta$ indexing a countable iid model with penalty $\L$. Then for any $\alpha \in [0, 1]$,
\begin{align}
& \E D_B(P, \Pth) \nonumber \\ & \leq \frac{1}{1-\alpha} \left[ \r_{\Theta, \L}^{(n)}(P) + \frac{2 \log \sum_{\theta \in \Theta} e^{-\tfrac{1}{2}\L(\theta)} A(P, \Pt)^{\alpha n}}{n} \right]. \label{eq:B1}
\end{align}
\end{corollary}

\begin{corollary}\label{cor:bhattacharyya-minus-penalty-pseudo-penalty}
Assume $X^n \iid P$, and let $\th$ be the penalized MLE over $\Theta$ indexing a countable iid model with penalty $\L$. Then for any $\alpha \in [0, 1]$,
\begin{align}
& \E D_B(P, \Pth) \nonumber \\ & \leq \frac{1}{1-\alpha} \left[ \r_{\Theta, \L}^{(n)}(P) + \frac{2 \log \sum_{\theta \in \Theta} A(P, \Pt)^{\alpha n} - \E \L(\th)}{n} \right]. \label{eq:B2}
\end{align}
\end{corollary}
For simplicity, the corollaries throughout this subsection will use $\alpha = 1/2$.

Consider a penalized MLE selected from an $\epsilon$-discretization of a continuous parameter space; as the sample size increases, one typically wants to shrink $\epsilon$ to make the grid more refined (see Section~\ref{sec:behavior-r}). Examining Corollaries~\ref{cor:bhattacharyya-pseudo-penalty} and~\ref{cor:bhattacharyya-minus-penalty-pseudo-penalty}, we see two opposing forces at work as $n$ increases: the grid-points themselves proliferate, while the $n$th power depresses the terms in the summation. An easy case occurs when $A(P, \Pt)$ is bounded by a Gaussian-shaped curve; we apply Corollary~\ref{cor:bhattacharyya-minus-penalty-pseudo-penalty} and invoke Lemma~\ref{lem:gaussian-summation-off-center}.
\begin{corollary}\label{cor:affinity-gaussian-decay}
Assume $X^n \iid P$, and let $\th$ be the penalized MLE over an $\epsilon$-discretization $\Te \subseteq \Theta \subseteq \R^d$ indexing an iid model with penalty $\L$. Assume $A(P, \Pt) \leq e^{-c \| \theta - \ts \|^2}$ for some $c > 0$ and some $\ts \in \Theta$. Then
\begin{align*}
& \E D_B(P, \Pth) \\ & \leq 2 \left[ \r_{\Te, \L}^{(n)}(P) + \frac{2d \log (1 + \tfrac{2\sqrt{2\pi}}{\epsilon \sqrt{n c}}) - \E \L(\th)}{n} \right].
\end{align*}
\end{corollary}

With $\epsilon$ proportional to $1/\sqrt{n}$, our bound on the summation of Hellinger affinities is stable. Corollary~\ref{cor:affinity-gaussian-decay-concrete} sets $\L = 0$ to demonstrate a more concrete instantiation of this result.
\begin{corollary}\label{cor:affinity-gaussian-decay-concrete}
Assume $X^n \iid P$, and let $\th$ be the MLE over an $\epsilon$-discretization $\Te \subseteq \Theta \subseteq \R^d$ indexing an iid model using $\epsilon = \sqrt{2/n}$. Assume $A(P, \Pt) \leq e^{-c \| \theta - \ts \|^2}$ for some $c > 0$ and some $\ts \in \Theta$. Then
\begin{align*}
\E D_B(P, \Pth) \leq 2 D(P \| \Te) + \frac{4d \log (1 + 4/\sqrt{c})}{n}.
\end{align*}
\end{corollary}

If $P$ is $\Pts$ in an exponential family with natural parameter indexed by $\Theta$, then Hellinger affinities do have a Gaussian-shaped bound as long as the minimum eigenvalue of the sufficient statistic's covariance matrix is uniformly bounded below by a positive number. We use the notation $\lambda_j(\cdot)$ for the $j$th largest eigenvalue of the matrix argument.
\begin{lemma}\label{lem:exponential-affinity-guassian-bound}
Let $\{ \Pt : \theta \in \Theta \subseteq \R^d \}$ be an exponential family with natural parameter $\theta$ and sufficient statistic $\phi$. Then
\begin{align*}
A(\Pts, \Pt) &\leq e^{-c \| \theta - \ts \|^2},
\end{align*}
where $c := \frac{1}{8} \inf_{\tt \in \Theta} \lambda_d (\C_{X \sim P_{\tt}} \phi(X))$.
\end{lemma}

In Lemma~\ref{lem:exponential-affinity-guassian-bound}, $c$ does not depend on $\ts$. If in addition the $\epsilon$-discretization is also a KL-net, then the risk of the estimator described in Corollary~\ref{cor:affinity-gaussian-decay-concrete} is uniformly bounded over data-generating distributions in $\Theta$. The minimax risk is no greater than the supremum risk of this particular estimator.
\begin{corollary}\label{cor:exponential-affinity-guassian-minimax}
Let $\Theta \subseteq \R^d$ index a set of distributions. Assume that for some $c > 0$, every $\ts \in \Theta$ has the property that $A(\Pts, \Pt) \leq e^{-c \| \theta - \ts \|^2}$. Assume further that there exists $\beta > 0$ such that for all $\epsilon > 0$, every $\epsilon$-discretization $\Te \subseteq \Theta$ is also a KL-net with balls of radius $\beta \epsilon^2$. Then the minimax Bhattacharyya risk of $\Theta$ has the upper bound
\begin{align*}
\min_{\th} \max_{\ts \in \Theta} \E_{X^n \iid \Pts} D_B(\Pts, \Pth) \leq \frac{4[\beta + d\log (1 + 4/\sqrt{c})]}{n}.
\end{align*}
\end{corollary}

In general, however, Hellinger affinity being uniformly bounded by a Gaussian curve may be too severe of a requirement. A weaker condition is to require only a power decay for $\theta$ far from some $\ts$.
\begin{theorem}\label{thm:gaussian-bound-center-power-bound-tail}
Assume $X^n \iid P$, and let $\th$ be the penalized MLE over an $\epsilon$-discretization $\Te \subseteq \Theta \subseteq \R^d$ indexing an iid model with penalty $\L$. Assume that for some $\ts, \tt^* \in \Theta$, radius $R$ and $a, c > 0$, the Hellinger affinity $A(P, \Pt)$ is bounded by $a/ \| \theta - \ts \|^b$ outside the ball $B(\ts, R)$ and bounded by $e^{-c \| \theta - \tt^* \|^2}$ inside the ball. If $R \geq 11a^{1/b} \vee 3\epsilon$, and $n \geq 2(d+1)/b$, then,
\begin{align*}
& \E D_B(P, \Pth) \\ & \leq 2 \Biggl[ \r_{\Te, \L}^{(n)}(P) + \\ & \qquad \frac{d[2 \log (1 + \frac{2\sqrt{2\pi}}{\epsilon\sqrt{n c}}) + 2\log(1 +\frac{4\sqrt{2}R}{\epsilon \sqrt{n b }}) ] + 3 - \E \L(\th)}{n} \Biggr].
\end{align*}
\end{theorem}

\begin{proof}
The part of the summation where Hellinger affinity is bounded by a Gaussian curve has the same bound as in Corollary~\ref{cor:affinity-gaussian-decay}, which is a direct consequence of Lemma~\ref{lem:gaussian-summation-off-center}.
\begin{align}\label{eq:gaussian-center}
\sum_{\theta \in \Te \cap B(\ts, R)} A(P, \Pt)^{\alpha n} &\leq \sum_{\theta \in \Te \cap B(\ts, R)} e^{-c \alpha n \| \theta - \tt^* \|^2}\nonumber\\
 &\leq \sum_{\theta \in \Te} e^{-c \alpha n \| \theta - \tt^* \|^2}\nonumber\\
 &\leq \left(1 + \frac{2\sqrt{\pi}}{\epsilon\sqrt{n \alpha c}} \right)^d.
\end{align}
Notice that the ``center" point for this Gaussian curve $\tt^*$ can be different from the center of the ball $\ts$.

The summation of the remaining terms is handled by Lemma~\ref{lem:power-decay-without-epsilon}, assuming $n \geq (d+1)/\alpha b$.
\begin{align}\label{eq:power-tail}
& \sum_{\theta \in \Te \cap B(\ts, R)^c} A(P, \Pt)^{\alpha n} \\ &\leq \sum_{\theta \in \Te \cap B(\ts, R)^c} \left(\frac{a}{\| \theta - \ts \|^{b}}\right)^{\alpha n}\nonumber\\
 &\leq \left( \frac{4R}{\epsilon \sqrt{n \alpha b \log(R/4a^{1/b})}} \right)^d\nonumber\\
 &\leq \left(1 + \frac{4R}{\epsilon \sqrt{n \alpha b \log(R/4a^{1/b})}} \right)^d.
\end{align}
The assumption that $R \geq 11 a^{1/b}$ assures us that $\log (R/4a^{1/b}) \geq 1$, simplifying the bound.

Each of (\ref{eq:gaussian-center}) and (\ref{eq:power-tail}) are at least $1$, so by Lemma~\ref{lem:log-sum}, the sum of their logs is bounded by the log of their sum plus $2\log 2$. Finally, substitute $\alpha = 1/2$.
\end{proof}

The sample size requirement in Theorem~\ref{thm:gaussian-bound-center-power-bound-tail} can be avoided by using a squared norm penalty. The bound we derive has superlinear order in the dimension.
\begin{theorem}\label{thm:gaussian-power-any-sample-size}
Assume $X^n \iid P$, and let $\th$ be the penalized MLE over an $\epsilon$-discretization $\Te \subseteq \Theta \subseteq \R^d$ indexing an iid model with penalty $\L(\theta) = \| \theta \|^2$. Assume that for some $\ts, \tt^* \in \Theta$, radius $R$ and $a, c > 0$, the Hellinger affinity $A(P, \Pt)$ is bounded by $a/ \| \theta - \ts \|^b$ outside the ball $B(\ts, R)$ and bounded by $e^{-c \| \theta - \tt^* \|^2}$ inside the ball. If $R \geq 11a^{1/b} \vee 3\epsilon$, then
\begin{align*}
\E D_B(P, \Pth) & \leq 2\r_{\Te, \|\cdot\|^2}^{(n)}(P) + \\ & \qquad \frac{4d\left[\log (1 + \frac{2\sqrt{2\pi}}{\epsilon\sqrt{n c}}) + \log(1 +\frac{29\sqrt{d} + 6R}{\epsilon \sqrt{n b }})\right]}{n} + \\ & \qquad\qquad \frac{4\log(2 + 2\tfrac{22}{R^3}) + 2\| \ts \|^2 + 8}{n}.
\end{align*}
\end{theorem}

\begin{proof}
This time we use Corollary~\ref{cor:bhattacharyya-pseudo-penalty} rather than Corollary~\ref{cor:bhattacharyya-minus-penalty-pseudo-penalty}. The challenge is to bound the summation
\begin{align*}
\sum_{\theta \in \Te \cap B(\ts, R)^c} e^{-\|\theta\|^2} \left( \frac{a}{\| \theta - \ts \|^b} \right)^{\alpha n}.
\end{align*}
Assuming $n \geq 2(d+1)/b$, we can bound that term as in Theorem~\ref{thm:gaussian-bound-center-power-bound-tail}. With smaller $n$, we invoke Lemmas~\ref{lem:power-decay-reversed-epsilon-condition} and~\ref{lem:power-decay-middle-epsilon-condition}. In each case, the bound is no greater than the one we have claimed.
\end{proof}

As in Corollary~\ref{cor:affinity-gaussian-decay}, the bounds in Theorems~\ref{thm:gaussian-bound-center-power-bound-tail} and~\ref{thm:gaussian-power-any-sample-size} remain stable if $\epsilon$ is proportional to $1/\sqrt{n}$.

As an example, we will see how these bound apply in a location family parametrized by the mean in $\Theta \subseteq \R^d$. First, we establish the power decay, assuming $P$ has a finite first moment. By Lemma~\ref{lem:hellinger-affinity-first-moments},
\begin{align*}
A(P, \Pt) &\leq \frac{2(s_P + s_\Theta)}{\| \theta - \ts\|},
\end{align*}
where $\ts := \E_{X \sim P} X$, and the other constants are the first central moments $s_P := \E_{X \sim P} \| X - \ts \|$ and $s_\Theta := \E_{X \sim \Pt} \| X - \theta \|$. Therefore, Theorems~\ref{thm:gaussian-bound-center-power-bound-tail} and~\ref{thm:gaussian-power-any-sample-size} apply if we can find a Gaussian-shaped Hellinger affinity bound that holds inside the ball centered at $\ts$ with radius $R = 22(s_P + s_\Theta) \vee 3\epsilon$.

In particular, let us assume the model comprises distributions that are continuous with respect to Lebesgue measure. Then we will also assume that $P$ is continuous; otherwise, the risk bound is infinite anyway. These assumptions ensure the existence of exact medians, enabling us to use Lemma~\ref{lem:hellinger-affinity-bound-marginal-median}.

Let $v$ be the vector of marginal medians of the model distribution with mean $\theta = 0$. The marginal median vector of \emph{any} model distribution $\Pt$ is then $\theta + v$. Let $m_P$ be the marginal median vector of $P$. By Lemma~\ref{lem:hellinger-affinity-bound-marginal-median}, for any $r \geq 0$, the inequality
\begin{align*}
A(P, \Pt) &\leq e^{-c \|\theta + v - m_P\|^2}
\end{align*}
holds for $\theta$ within $B(m_P - v, r)$, where $c$ is $\tfrac{1}{2d}$ times the minimum squared marginal density of $\Pt$ within $r$ of its median. It remains to identify an $r$ large enough that $B(m_P - v, r)$ contains $B(\ts, R)$. Using the triangle inequality and then Lemma~\ref{lem:marginal-meadian-mean-distance} to bound the distance between means and medians,
\begin{align*}
\|\theta - (m_P - v)\| &= \|\theta - \ts + v - (m_P - \ts)\|\\
 &\leq \| \theta - \ts \| + \| v \| + \| m_P - \ts \|\\
 &\leq \| \theta - \ts \| + s_\Theta \sqrt{d} + s_P \sqrt{d}.
\end{align*}
For $\theta$ in the ball $B(\ts, R)$, the first term is bounded by $R$. This tells us that the ball $B(m_p - v, R + \sqrt{d}(s_\Theta + s_P))$ contains $B(\ts, R)$.

Thus if all the marginal densities of $\Pt$ are positive within $R + \sqrt{d}(s_\Theta + s_P)$ of their medians, then there is a positive $c$ for which
\begin{align*}
A(P, \Pt) &\leq e^{-c \|\theta - (m_P - v)\|^2}
\end{align*}
in $B(\ts, R)$, confirming that Theorems~\ref{thm:gaussian-bound-center-power-bound-tail} and~\ref{thm:gaussian-power-any-sample-size} hold.

If the data-generating distribution is itself in the location family, then $P = \Pts$ and $s_P = s_\Theta$. Thus the bound holds uniformly over $\ts \in \Theta$. If there exists $\beta > 0$ such that every $\epsilon$-discretization of the family is a KL-net with radius $\beta \epsilon^2$, then a minimax risk bound can be derived in the same manner as Corollary~\ref{cor:exponential-affinity-guassian-minimax}.

\subsection{Log reciprocal pmf of $\mathbf{\th}$ as pseudo-penalty}\label{sec:log-reciprocal-pmf-pseudo-penalty}

In Section~\ref{sec:bhattacharyya-pseudo-penalty}, we chose a pseudo-penalty to have an expectation that easy to handle; we only had to worry about the resulting log summation. Now we will select a pseudo-penalty with the opposite effect. We can eliminate Corollary~\ref{cor:subtract-penalty}'s log summation term by letting $L$ be twice a codelength function. The smallest resulting $\E L(\th)$ comes from setting $L$ to be two times the log reciprocal of the probability mass function of $\th$. This expectation is the Shannon entropy $H$ of the penalized MLE's distribution (i.e. the image measure of $P$ under the $\Theta$-valued deterministic transformation $\th$).
\begin{corollary}\label{cor:bound-with-entropy-term}
Let $X^n \iid P$, and let $\th$ be a penalized MLE over all $\theta \in \Theta$ indexing a countable iid model. Then
\begin{align*}
\E D_B(P, \Pth) \leq \r_{\Theta, \L}^{(n)}(P) + \frac{2 H(\th) - \E \L(\th)}{n}.
\end{align*}
\end{corollary}

It is known that the risk of the MLE is bounded by the log-cardinality of the model (e.g. \cite{li1999}); Corollary~\ref{cor:bound-with-entropy-term} implies a generalization of this fact for penalized MLEs:
\begin{align*}
\E D_B(P, \Pth) \leq \r_{\Theta, \L}^{(n)}(P) + \frac{2 \log |\Theta| - \E \L(\th)}{n}.
\end{align*}
Importantly, Corollary~\ref{cor:bound-with-entropy-term} also applies to models of infinite cardinality.

\begin{lemma}\label{lem:estimator-entropy-bound}
Let $\Te \subseteq \R^d$ be an $\epsilon$-discretization, and let $\th$ be a $\Te$-valued random vector. Suppose that for some $\ts \in \R^d$ and some radius $R \geq 0$, every $\theta \in \Te$ outside of $B(\ts, R)$ has probability bounded by $e^{-c \| \theta - \ts\|^2}$. Then the entropy of $\th$ has the bound
\begin{align*}
H(\th) \leq \frac{d}{2} \left( \frac{4\sqrt{\pi}}{\epsilon \sqrt{c}} \right)^d + d \log (1+\tfrac{2[c^{-1/2} \vee R \vee 3\epsilon]}{\epsilon}).
\end{align*}
\end{lemma}

If $\epsilon \sqrt{c} \leq 4 \sqrt{\pi}$, then this bound grows exponentially in $d$. However, if $c$ and $R$ are known, then one can set $\epsilon \geq 4 \sqrt{\pi} /(\sqrt{c} \vee R/3)$ and find that $H(\th)$ is guaranteed to be bounded by $3d$. Of course, one needs to take the behavior of the index of resolvability into account as well; good overall behavior will typically require that $c$ has order $n$.

In certain models satisfying $D(\Pt \| \Pts) \geq a\| \theta - \ts\|^2$ for some $a > 0$, we surmise that it may be possible to establish the applicability of Lemma~\ref{lem:estimator-entropy-bound} (with $c$ having order $n$) by using information theoretic large deviation techniques along the lines of \cite[Thm 19.2]{grunwald2007}.

\subsection{Quadratic form as pseudo-penalty}\label{sec:quadratic-pseudo-penalty}

Other simple corollaries come from using a quadratic pseudo-penalty $L(\theta) = (\theta - \E \th)' M (\theta - \E \th)$ for some positive definite matrix $M$. The expected pseudo-penalty is then
\begin{align*}
\E L(\th) &= \tr M \C \th,
\end{align*}
where $\C \th$ denotes the covariance matrix of the random vector $\th(X^n)$ with $X^n \iid P$. For the log summation term, we note that
\begin{align*}
\sum_{\te \in \Te} e^{-(\te - \E \th)' M (\te - \E \th)} &\leq \sum_{\te \in \Te} e^{-\lambda_d(M)\|\te - \E \th\|^2}\\
 &\leq \left(1 + \frac{2\sqrt{\pi}}{\epsilon \sqrt{\lambda_d(M)}} \right)^d,
\end{align*}
by Lemma~\ref{lem:gaussian-summation-off-center}. Using $\alpha I_d$ as $M$ gives us Corollary~\ref{cor:bound-with-quadratic-pseudo-identity}.

\begin{corollary}\label{cor:bound-with-quadratic-pseudo-identity}
Assume $X^n \iid P$, and let $\th$ be the penalized MLE over an $\epsilon$-discretization $\Te \subseteq \Theta \subseteq \R^d$ indexing an iid model with penalty $\L$. Then for any $\alpha \geq 0$,
\begin{align*}
\E D_B(P, \Pth) & \leq \r_{\Te, \L}^{(n)}(P) + \\ & \qquad \frac{2 d \log(1 + \tfrac{2\sqrt{\pi}}{\epsilon\sqrt{\alpha}}) + \alpha \V \th - \E \L(\th)}{n}.
\end{align*}
\end{corollary}

As described in Section~\ref{sec:behavior-r}, one gets desirable order $1/n$ behavior from $\r_{\Te, \L}^{(n)}(P)$ by using $\epsilon$ proportional to $1/\sqrt{n}$. For either of these two corollaries above to have order $1/n$ bounds, the numerator of the second term should be stable in $n$. In Corollary~\ref{cor:bound-with-quadratic-pseudo-identity}, one sets $\alpha$ proportional to $1/\epsilon^2$ and thus needs $\V \th$ to have order $1/n$. In many cases, such as ordinary MLE with an exponential family, the covariance matrix of the optimizer over $\Theta$ is indeed bounded by a matrix divided by $n$. However, one still needs to handle the discrepancy in behavior between the continuous and discretized estimator.

In a sense, Corollary~\ref{cor:bound-with-quadratic-pseudo-identity} shifts the problem to another risk-related quantity, while the pseudo-penalties used in Sections~\ref{sec:bhattacharyya-pseudo-penalty} and \ref{sec:log-reciprocal-pmf-pseudo-penalty} provide more direct ways of deriving exact risk bounds of order $1/n$.

\subsection{Penalty as pseudo-penalty}

Another simple corollary to Theorem~\ref{thm:arbitrary-penalty} uses $L = \alpha \L$.
\begin{corollary}\label{cor:bound-with-penalty-pseudo}
Assume $X^n \iid P$, and let $\th$ be the penalized MLE over $\Theta$ indexing a countable iid model with penalty $\L$. Then
\begin{align*}
\E D_B(P, \Pth) & \leq \r_{\Theta, \L}^{(n)}(P) + \\ & \qquad \frac{2 \log \sum_{\theta \in \Theta} e^{-\tfrac{\alpha + 1}{2}\L(\theta)} + \alpha \E \L(\th)}{n}.
\end{align*}
\end{corollary}

Bayesian MAP (maximum a posteriori) is a common penalized likelihood procedure that has insufficient penalty for the index of resolvability bound (\ref{eq:usual}) to be valid. In that case, Corollary~\ref{cor:set-l-zero} applies (where $\L$ comprises the logs of the reciprocals of prior masses), but the sum of exponential terms may be infinite. An alternative approach comes from Corollary~\ref{cor:bound-with-penalty-pseudo} by setting $\alpha = 1$.

\begin{corollary}\label{cor:bound-map}
Assume $X^n \iid P$, and let $\th$ be the MAP estimate over $\Theta$ indexing a countable iid model with prior pmf $q$. Then
\begin{align*}
\E D_B(P, \Pth) \leq \r_{\Theta, \log 1/q}^{(n)}(P) + \frac{\E \log (1/q(\th))}{n}.
\end{align*}
\end{corollary}

For $\epsilon$-discretizations, realize that $q$ has to change as the refinement increases; thus the second term in Corollary~\ref{cor:bound-map} should be considered to have order strictly larger than $1/n$ in that context.

\section{Adaptive Modeling}\label{sec:adaptive-modeling}

Suppose $\Theta = \bigcup_{k \geq 1} \Theta^{(k)}$ is a model class and each $\Theta^{(k)}$ is a model of countable cardinality. Let us index the distributions in $\Theta$ by $\nu = (k, \theta)$ with $\theta \in \Theta^{(k)}$. Assume the penalty and pseudo-penalty have the form $\L(\nu) = \L_0(k) + \L_k(\theta)$ and $L(\nu) = L_0(k) + L_k(\theta)$. Then Theorem~\ref{thm:arbitrary-penalty} can be useful if the penalty plus pseudo-penalty on $k$ is large enough to counteract the within-model summations.
\begin{align*}
& \sum_{\nu \in \Theta} e^{-\tfrac{1}{2}[\L(\nu) + L(\nu)]}  \\ &= \sum_{k \geq 1} \left[ e^{-\tfrac{1}{2}[\L_0(k) + L_0(k)]} \sum_{\theta \in \Theta^{(k)}} e^{-\tfrac{1}{2}[\L_k(\theta) + L_k(\theta)]} \right].
\end{align*}
One can use $L_0(k) = 0$ to avoid having to worry about the behavior of $\hat{k}$. Then bounds on $\sum_{\theta \in \Theta^{(k)}} e^{-[\L_k(\theta) + L_k(\theta)]}$ should be known so that one can devise a penalty on $k$ that bounds the weighted sum of these summations. In particular, if such bounds do not depend on any unknown quantities, then one can set $\L_0(k) \geq k \sqrt{2} + 2\log \sum_{\theta \in \Theta^{(k)}} e^{-\tfrac{1}{2}[\L_k(\theta) + L_k(\theta)]}$ and have
\begin{align*}
\log \sum_{k \geq 1} \left[ e^{-\tfrac{1}{2}\L_0(k)} \sum_{\theta \in \Theta^{(k)}} e^{-\tfrac{1}{2}[\L_k(\theta) + L_k(\theta)]} \right] &\leq 0.
\end{align*}
It remains to deal with $\E L_k(\th)$, either by bounding it or by absorbing it into the risk as in Corollary~\ref{cor:bhattacharyya-pseudo-penalty}.

An important feature of the resolvability bound method is its generality; bounds can be derived that assume very little about the data-generating distribution. In non-adaptive models, however, the bound cannot become small if the data-generating distribution is far from Gaussian. Our hope is to derive similar exact risk bounds for penalized MLEs over flexible model classes as well, such as Gaussian mixtures, so that $D(P \| \Theta)$ can be made small (or possibly zero) for large classes of potential data-generating distributions.

\section*{Acknowledgment}

We would like to acknowledge our advisor Andrew R. Barron for teaching us about the MDL method for risk bounds. We also thank him for his guidance in preparing this document, especially for driving us toward the most important aspects and consequences of our ideas.

We are also grateful to the anonymous reviewers whose detailed feedback led to dramatic improvements to this paper. Most importantly, they brought \cite{zhang2006} to our attention, which inspired us to write Section~\ref{sec:bhattacharyya-pseudo-penalty}.

\section*{Appendix}\label{sec:appendix}

\subsection{Miscellaneous facts}

The following handy facts are known, but we provide brief proofs here nonetheless.

\begin{lemma}\label{lem:squared-norm-bound}
For any vectors $u,v$ in a real inner product space,
\begin{align*}
\|u - v\|^2 \leq 2\|u\|^2 + 2\|v\|^2.
\end{align*}
\end{lemma}

\begin{proof}
We apply the Cauchy-Schwarz inequality followed by the arithmetic-geometric mean inequality.
\begin{align*}
\|a - b\|^2 &= \|a\|^2 + \|b\|^2 - 2\langle a, b \rangle\\
 &\leq \|a\|^2 + \|b\|^2 + 2 \|a\| \|b\|\\
 &\leq \|a\|^2 + \|b\|^2 + 2 (\|a\|^2/2 + \|b\|^2/2).
\end{align*}
\end{proof}

\begin{lemma}\label{lem:quadratic-form-bound}
Let $v \in \R^d$, and let $M$ be a symmetric $d \times d$ matrix. Then
\begin{align*}
\lambda_d(M) \leq \frac{v' M v}{\| v \|^2} \leq \lambda_1(M).
\end{align*}
\end{lemma}

\begin{proof}
Any symmetric matrix has an orthonormal eigenvector decomposition $M = Q \Lambda Q'$.
\begin{align*}
v' M v &= v' Q \Lambda Q' v\\
 &= \sum_j \lambda_j (Q'v)_j^2\\
 &= \| v \|^2 \sum_j \lambda_j \left( Q'\frac{v}{\| v \|} \right)_j^2.
\end{align*}
Realize that squared values in the summation are eigenvector-basis coordinates of the unit vector in the direction of $v$. As such, these squared coordinates must sum to $1$. Thus the summation is a weighted average of the eigenvalues. It achieves its maximum $\lambda_1$ when $v$ is in the direction of the first eigenvector, and it achieves its minimum $\lambda_d$ when $v$ is in the direction of the last eigenvector.
\end{proof}

\begin{lemma}\label{lem:log-sum}
Let $a_1, \ldots, a_K \geq 1/K$. Then
\begin{align*}
\log \sum_k a_k &\leq \sum_k \log a_k + K \log K.
\end{align*}
\end{lemma}

\begin{proof}
We apply the log-sum inequality and realize that it produces coefficients bounded by $1$.
\begin{align*}
\log \sum_k a_k &= \frac{1}{\sum_k K a_k} \left[\left( \sum_k K a_k \right) \log \frac{\sum_k K a_k}{\sum_k 1} \right]\\
 &\leq \frac{1}{\sum_k K a_k} \left[ \sum_k K a_k \log \frac{K a_k}{1} \right]\\
 &\leq \sum_k \log K a_k\\
 &= \sum_k \log a_k + K \log K.
\end{align*}
\end{proof}

\begin{lemma}\label{lem:marginal-meadian-mean-distance}
Let $X \sim P$, where $P$ is a probability distribution on $\R^d$ with marginal median vector $m_P$. Then
\begin{align*}
\| m_P - \E X\| &\leq \sqrt{d} \, \E \| X - \E X\|.
\end{align*}
\end{lemma}

\begin{proof}
Superscripts indicate coordinates. We use subadditivity of square root and the fact that the median minimizes expected absolute deviation.
\begin{align*}
\| \E X - m_P \|  &\leq \| \E X - m_P \|_1\\
 &= \sum_{j=1}^d | \E X^{(j)} - m_P^{(j)}|\\
 &= \sum_{j=1}^d |\E (X^{(j)} - m_P^{(j)})|\\
 &\leq \sum_{j=1}^d \E | X^{(j)} - m_P^{(j)}|\\
 &\leq \sum_{j=1}^d \E | X^{(j)} - \E X^{(j)}|\\
 &\leq \sqrt{d} \, \E \| X - \E X \|.
\end{align*}
We used the fact that $l^1$ and $l^2$ satisfy $\| v \| \leq \| v \|_1 \leq \sqrt{d} \| v\|$.
\end{proof}

\subsection{Jensen differences}
For any random vector $Y$ and any function $f$, we will call $\E f(Y) - f(\E Y)$ a \emph{Jensen difference}.

\begin{lemma}\label{lem:jensen-difference-bounds}
Let $Y$ be a random vector with convex support $S \subseteq \R^d$. If $f: \R^d \rightarrow \R$ is twice continuously differentiable, then
\begin{align*}
\inf_{y \in S} \lambda_d( \nabla \nabla' f(y)) \leq \frac{\E f(Y) - f(\E Y)}{\V Y / 2} \leq \sup_{y \in S} \lambda_1( \nabla \nabla' f(y)).
\end{align*}
\end{lemma}

\begin{proof}
We start with a second-order Taylor expansion with Lagrange remainder.
\begin{align*}
f(Y) &= f(\E Y) + (Y - \E Y)' \nabla f(\E Y) + \\ & \qquad \tfrac{1}{2} (Y - \E Y)' \nabla \nabla' f(\tilde{Y}) (Y - \E Y),
\end{align*}
for some $\tilde{Y}$ on the segment from $Y$ to $\E Y$. By Lemma~\ref{lem:quadratic-form-bound}, the quadratic form has the bounds
\begin{align*}
\| Y - \E Y \|^2 \lambda_d (\nabla \nabla' f(\tilde{Y})) & \leq (Y - \E Y)' \nabla \nabla' f(\tilde{Y}) (Y - \E Y) \\ & \leq \| Y - \E Y \|^2 \lambda_1 (\nabla \nabla' f(\tilde{Y})).
\end{align*}
The smallest and largest eigenvalues of the Hessian at $\tilde{Y}$ are bounded by the infimum of smallest eigenvalue and supremum of largest eigenvalue taken over the support of $Y$.
\begin{align*}
& \| Y - \E Y \|^2 \inf_{y \in S} \lambda_d (\nabla \nabla' f(y)) \\ & \leq (Y - \E Y)' \nabla \nabla' f(\tilde{Y}) (Y - \E Y) \\ & \leq \| Y - \E Y \|^2 \sup_{y \in S} \lambda_1 (\nabla \nabla' f(y)).
\end{align*}
Substituting this second-order Taylor expansion into $\E f(Y) - f(\E Y)$ gives the desired result.
\end{proof}

\subsection{Infimum on a grid}\label{sec:infimum-grid}

In many cases we will need to ensure that the infimum of a function on a grid of its domain approaches the overall infimum as the grid becomes increasingly refined. Lemma~\ref{lem:hessian-eigenvalue} will prove to be remarkably useful for such tasks.

\begin{lemma}\label{lem:hessian-eigenvalue}
Let $\Te \subseteq \Theta \subseteq \R^d$, and assume $f: \Theta \rightarrow \R$ is twice continuously differentiable. If $\theta$ is in the convex hull of $\Te \cap B(\theta, \delta)$, then
\begin{align*}
\inf_{\te \in \Te} f(\te) \leq f(\theta) + \frac{\delta^2}{2} \sup_{\tilde{\theta} \in B(\theta, \delta)} \lambda_1(\nabla \nabla' f(\tilde{\theta}))_+.
\end{align*}
\end{lemma}


\begin{proof}
We first bound the infimum over $\Te$ by the infimum over $\Te \cap B(\theta, \delta)$. Then that infimum is bounded by the expectation using any distribution $Q$ on those grid-points. We have assumed that $\theta$ is some weighted average of nearby grid-points (the ones at most $\delta$ distance away), and we can use that same weighted averaging to define $Q$. Then the expectation of the random selection is $\theta$, and we apply Lemma~\ref{lem:jensen-difference-bounds}.
\begin{align*}
& \inf_{\te \in \Te} f(\te) \\ & \leq \inf_{\te \in \Te \cap B(\theta, \delta)} f(\te)
\\ & \leq \E_{\te \sim Q} f(\te)
\\ &\leq f(\E_{\te \sim Q} \te) + \\ & \qquad \tfrac{1}{2} \E_{\te \sim Q} \| \te - \E_{\te \sim Q} \te\|^2 \sup_{\tt \in B(\theta, \delta)} \lambda_1 (\nabla \nabla' f(\tt))
\\ &\leq f(\theta) + \tfrac{1}{2} \delta^2 \sup_{\tt \in B(\theta, \delta)} \lambda_1 (\nabla \nabla' f(\tt)),
\end{align*}
assuming $\lambda_1( \nabla \nabla' f(\tt))$ is non-negative. If the maximum eigenvalue is negative, i.e. if $f$ is strictly concave within the ball, then the second order term is upper bounded by zero.
\end{proof}

Suppose $\Te \subseteq \Theta \subseteq \R^d$ is an $\epsilon$-discretization, as defined in Section~\ref{sec:discrete}. If $\Theta$ is convex, then every $\theta$ in the convex hull of $\Te$ satisfies the conditions of Lemma~\ref{lem:hessian-eigenvalue} with $\epsilon\sqrt{d}$ as $\delta$. In particular, if every dimension of $\Theta$ is either $\R$ or a closed half-line, then there is an obvious $\epsilon$-discretization that makes Lemma~\ref{lem:hessian-eigenvalue} apply for every $\theta \in \Theta$. For less favorably shaped $\Theta$, one can consider adding more grid-points ``on top of'' an $\epsilon$-discretization.

\subsection{Behavior of $\mathbf{\r_{\Te, \L}^{(n)}(P)}$}\label{sec:behavior-r}

One way to bound $\r_{\Te, \L}^{(n)}(P)$ is to use an approach similar to Section~\ref{sec:infimum-grid}. Suppose $p_\theta(x)$ is twice continuously differentiable in $\theta$. We define a type of Fisher ``cross-information" matrix
\begin{align*}
I_P(\tt) := \E_{X \sim P} \nabla \nabla' \left[\log \frac{1}{\pt(X)} \right]_{\theta = \tt},
\end{align*}
where the Hessian is taken with respect to $\theta$. Note that if $\pt$ represents an exponential family, then $P$ does not play a role. In that case, $I_P(\tt)$ reduces to the ordinary Fisher information matrix.

Let $B(\theta, \delta)$ denote the closed Euclidean ball centered at $\theta$ with radius $\delta$, and let $\lambda_j(\cdot)$ denote the $j$th largest eigenvalue of its matrix argument.
\begin{theorem}\label{thm:fisher-cross-info}
Let $\Te \subseteq \Theta \subseteq \R^d$. Assume that $\L:\Theta \rightarrow \R$ is twice continuously differentiable and that $\pt(x)$ is twice continuously differentiable in $\theta$ for every fixed $x$ in its domain. If $\theta \in \Theta$ is in the convex hull of $\Te \cap B(\theta, \delta)$, then 
\begin{align*}
\r_{\Te, \L}^{(n)}(P) & \leq D(P \| \Pt) + \frac{\L(\theta)}{n} + \\ & \qquad \frac{\delta^2}{2} \sup_{\tt \in B(\theta, \delta)} \lambda_1 (I_P(\tt) + \tfrac{1}{n}\nabla \nabla' \L(\tt))_+ .
\end{align*}
\end{theorem}

\begin{proof}
Define $f_X(\theta) := \log \frac{p(X)}{\pt(X)} + \frac{\L(\theta)}{n}$, and let $X \sim P$. We use a second-order Taylor expansion at $\theta$ with Lagrange remainder and reason similarly to the proofs of Lemmas~\ref{lem:jensen-difference-bounds} and~\ref{lem:hessian-eigenvalue}.
\begin{align*}
\r_{\Te, \L}^{(n)}(P) &= \inf_{\te \in \Te} \E f_X(\te)\\
 &= \inf_{\te \in \Te} \E \Biggl( f_X(\theta) + (\te - \theta)' \nabla f_X(\theta) + \\ & \qquad \tfrac{1}{2}(\te - \theta)' [\nabla \nabla' f_X(\tt)] (\te - \theta) \Biggr)\\
 &= \inf_{\te \in \Te} \Biggl( \E f_X(\theta) + (\te - \theta)' \E \nabla f_X(\theta) + \\ & \qquad \tfrac{1}{2}(\te - \theta)' [\E \nabla \nabla' f_X(\tt)] (\te - \theta) \Biggr),
\end{align*}
for some $\tt$ between $\theta$ and $\te$.

The infimum is bounded by the expectation for any random $\te$ on the grid-points. In particular, use the distribution on neighboring grid-points that makes $\te$ have expectation $\theta$. The first-order term is elminated, while the second-order term is bounded by half the expected squared length of the vector $\te - \theta$ times the largest eigenvalue (if positive).
\end{proof}

When $\Te \subseteq \Theta$ is an $\epsilon$-discretization, we use $\epsilon \sqrt{d}$ as $\delta$.

\begin{corollary}\label{cor:index-of-resolvability}
Let $\Theta \subseteq \R^d$ be a convex parameter space having densities twice continuously differentiable in $\theta$. Let $\Te \subseteq \Theta$ be an $\epsilon$-discretization. For any $\theta$ in the convex hull of $\Te$,
\begin{align*}
\r_{\Te, \L}^{(n)}(P) &\leq D(P \| \Pt) + \frac{\L(\theta)}{n} + \\ & \qquad \frac{\epsilon^2 d}{2} \sup_{\tt \in B(\theta, \epsilon \sqrt{d})} \lambda_1 (I_P(\tt) + \tfrac{1}{n}\nabla \nabla' \L(\tt))_+ .
\end{align*}
\end{corollary}

If one uses discretization $\epsilon = a/\sqrt{n}$,
\begin{align*}
\r_{\Te, \L}^{(n)}(P) &\leq D(P \| \Pt) + \frac{\L(\theta) + a^2dz/2}{n},
\end{align*}
with $z := \sup_{\tt \in B(\theta, \sqrt{ad})} \lambda_1 (I_P(\tt) + \nabla \nabla' \L(\tt))_+$ which does not depend on $n$. Notice that this bound uses the $n=1$ version of the supremum term, because they cannot increase with $n$. Notice also that, in general, $z$ will increase with $d$. One could set $a^2 = 1/d$ to cancel out all dimension dependence, but that has an undesirable overall effect on the risk bound results put forward in this paper.

One will most likely want to invoke these results with $\Pt$ being the rI-projection of $P$ onto $\Theta$ if it exists. In particular, if $P$ is in the model, then we can let $\Pt$ be $P$ to get an exact bound of order $1/n$  for $\r_{\Te, \L}^{(n)}(P)$.

\subsection{Bounding summations over grid-points}\label{sec:gaussian-summation}

Lemmas~\ref{lem:gaussian-summation} and~\ref{lem:gaussian-summation-off-center} provide bounds for summations of Gaussian-shaped functions over $\epsilon$-discretizations of $\R^d$.

\begin{lemma}\label{lem:gaussian-summation}
Let $\Te$ be an $\epsilon$-discretization of $\R^d$. Then for any $c > 0$ and $v \in \Te$,
\begin{align*}
\sum_{\theta \in \Te} e^{-c \| \theta - v \|^2} \leq \left( 1 + \tfrac{\sqrt{\pi}}{\epsilon \sqrt{c}} \right)^d.
\end{align*}
\end{lemma}

\begin{proof}
We can assume without loss of generality that $v$ is the zero vector and that $\Te$ includes zero. First, consider the one-dimensional problem. The ``center" term equals $1$ and the sum of the other terms is bounded by a Gaussian integral.
\begin{align*}
\sum_{\theta \in \Te} e^{- c \theta^2} &= \sum_{\theta \in \Te} e^{-c \epsilon^2 (\theta/\epsilon)^2}\\
 &= \sum_{z \in \Z} e^{-c \epsilon^2 z^2}\\
 &\leq 1 + \int_\R e^{- c\epsilon^2 z^2} dz\\
 &= 1 + \frac{\sqrt{\pi}}{\epsilon \sqrt{c}}.
\end{align*}

The $d$-dimensional problem can be bounded in terms of $d$ instances of the one-dimensional problem. Let $\Te^{(1)}, \ldots, \Te^{(d)}$ represent the underlying discretizations of $\R$, so that $\Te = \prod_j \Te^{(j)}$.
\begin{align*}
\sum_{\theta \in \Te} e^{- c \|\theta\|^2} &= \sum_{\theta \in \Te} e^{- c \sum_j \theta_j^2}\\
 &= \sum_{\theta_1 \in \Te^{(1)}} \ldots \sum_{\theta_d \in \Te^{(d)}} \prod_j e^{- c \theta_j^2}\\
 &= \prod_j \sum_{\theta_j \in \Te^{(j)}} e^{- c \theta_j^2}\\
 &\leq \prod_j \left( 1 + \tfrac{\sqrt{\pi}}{\epsilon \sqrt{c}} \right)\\
  &= \left( 1 + \tfrac{\sqrt{\pi}}{\epsilon \sqrt{c}} \right)^d.
\end{align*}
\end{proof}

Similar reasoning provides a slightly larger bound if the peak of the Gaussian function is not necessarily in the discretization.
\begin{lemma}\label{lem:gaussian-summation-off-center}
Let $\Te$ be an $\epsilon$-discretization of $\R^d$. Then for any $c > 0$ and $v \in \R^d$,
\begin{align*}
\sum_{\theta \in \Te} e^{-c \| \theta - v \|^2} \leq \left( 1 + \tfrac{2\sqrt{\pi}}{\epsilon \sqrt{c}} \right)^d.
\end{align*}
\end{lemma}

\begin{proof}
Again, we begin with the one-dimensional problem. The closest point to $v$ contributes at most $1$ to the sum. We reduce to Lemma~\ref{lem:gaussian-summation} by comparison to $\Theta_{\epsilon/2}^*$, the $(\epsilon/2)$-grid that includes $v$. Each point on the original grid can be translated ``inward'' to a neighboring point on the new (more refined) grid. The sum over the new grid's points will be larger than the sum over the original grid's points.
\begin{align*}
\sum_{\theta \in \Te} e^{- c (\theta - v)^2} &\leq \sum_{\theta \in \Theta_{\epsilon/2}^*} e^{- c (\theta - v)^2}\\
 &\leq 1 + \frac{\sqrt{\pi}}{(\epsilon/2)\sqrt{c}}.
\end{align*}

As before, the $d$-dimensional problem reduces to the one-dimensional problem.
\begin{align*}
\sum_{\theta \in \Te} e^{- c \|\theta - v\|^2} &= \prod_j \sum_{\theta_j \in \Te^{(j)}} e^{- c (\theta_j - v_j)^2}\\
 &\leq \prod_j \left( 1 + \tfrac{2\sqrt{\pi}}{\epsilon \sqrt{c}} \right)\\
  &= \left( 1 + \tfrac{2\sqrt{\pi}}{\epsilon \sqrt{c}} \right)^d.
\end{align*}
\end{proof}

There are important situations in which a function can be bounded by one type of behavior near its peak and by another type of behavior further away. When the tail behavior has a spherically symmetric bound, Lemma~\ref{lem:tail-summation} can help us convert the summation problem into a one-dimensional integral.

\begin{lemma}\label{lem:tail-summation}
Let $f$ be a real-valued function of the form $f(\theta) = g(\| \theta - \ts \|)$ for some non-increasing and non-negative function $g$. Let $\Te$ be an $\epsilon$-discretization of $\Theta \subseteq \R^d$. For any radius $R \geq 3\epsilon$,
\begin{align*}
\sum_{\theta \in \Te \cap B(\ts, R)^c} f(\theta) &\leq \frac{2\pi^{d/2}}{(\epsilon/4)^d\Gamma(d/2)} \int_{R/4}^\infty g(r) r^{d-1} dr,
\end{align*}
if the integral is well-defined.
\end{lemma}

\begin{proof}
First, we bound the summation outside the ball with diameter $2R$ by the summation outside the coordinate-axes-aligned hypercube inscribed by the ball, which has sides of length $\sqrt{2} R$. Next, we introduce a new more refined grid $\Theta_{\epsilon/2}^*$, which is the $\epsilon/2$-discretization that includes the point $\ts$. Consider the hyperplanes of $\Te$ grid-points orthogonal to the first coordinate axis. One can ``translate inward" each of these hyperplanes to a hyperplane in $\Theta_{\epsilon/2}^*$ that is closer to $\ts$. The argument can be repeated for each coordinate axis in turn. Because have assumed $R \geq 3\epsilon$, each \emph{translated} point outside of the $\sqrt{2} R$ hypercube remains outside of the $R$ hypercube. Thus it suffices to sum over the points of $\Theta_{\epsilon/2}^*$ outside of the hypercube of side-length $R$. For the remainder of this proof, we will assume without loss of generality that $\ts$ is the zero vector.

We will complete our proof by bounding the function's values at each point by its average value over a unique hypercube closer to zero. Given the standard $\epsilon$-discretization of $\R^d$, let $j$ index \emph{shells} radiating outward from the origin. The $j$th shell comprises the grid-points on the boundary of the centered hypercube of side-length $2j\epsilon$, along with the boundary hypercubes of volume $\epsilon^d$ demarcated by those grid-points. (We will consider the origin point itself to be the $0$th shell.) The total number of points in the first $j$ shells is $(2j+1)^d$, so the number of points in the $(j+1)$st shell is $[2(j+1)+1]^d - [2j+1]^d$. Similarly, the number of hypercubes in the $(j+1)$st shell is $[2(j+1)]^d - [2j]^d$. Because $t \mapsto t^d$ is convex and increasing on $\R^+$,
\begin{align*}
[2(j+1)+1]^d - [2j+1]^d &\leq [2(j+1)+2]^d - [2j+2]^d\\
 &= [2(j+2)]^d - [2(j+1)]^d.
\end{align*}
In other words, the number of points in the $(j+1)$st shell is no greater than the number of hypercubes in the $(j+2)$nd shell.

Finally, we introduce yet another grid, $\Theta_{\epsilon/4}^*$. We know that the number of points in the $j$th shell of $\Theta_{\epsilon/2}^*$ is bounded by the number of hypercubes in the $(j+1)$st shell of $\Theta_{\epsilon/4}^*$. If we can establish that these hypercubes are closer to the origin than are the points in the $j$th shell of $\Theta_{\epsilon/2}^*$, then we can bound the sum of the points' function values by the sum of the hypercubes' \emph{average} function values.

The points comprising the $j$th shell of $\Theta_{\epsilon/2}^*$ are inscribed by a sphere of radius $j\epsilon/2$; that sphere is inscribed by a hypercube of radius $\tfrac{j\epsilon/2}{\sqrt{2}}$. As $j$ increases, the $(j+1)$st shell of $\Theta_{\epsilon/4}^*$ will be about half as far from the origin as the the $j$th shell of $\Theta_{\epsilon/2}^*$. Because we have assumed $R \geq 3\epsilon$, the smallest $j$ we will need to worry about is $j=3$. The third shell of $\Theta_{\epsilon/2}^*$ has distance $3\epsilon/2$ from the origin, while the $4$th shell of $\Theta_{\epsilon/4}^*$ has distance $\epsilon$ from the origin. As $\epsilon < \tfrac{3\epsilon/2}{\sqrt{2}}$, \emph{every hypercube} in the $4$th shell of $\Theta_{\epsilon/4}^*$ is entirely closer to the origin than \emph{any point} in the $3$rd shell of $\Theta_{\epsilon/2}^*$. The comparison continues to hold for all $j \geq 3$.

The average value of $f$ on a hypercube within $\Theta_{\epsilon/4}^*$ is equal to the integral over that region divided by the hypervolume $(\epsilon/4)^d$. The inner-most hypercube that we need to consider is a distance of $R/4$ from the origin. We let $H_0(z)$ denote the coordinate-axes-aligned hypercube centered at the origin with side-length $2z$; we will need to integrate over the complement of this hypercube. Because $g/(\epsilon/4)^d$ is non-negative, we can bound this integral by the integral over a larger region, the complement of a ball. We then use spherical symmetry to reduce the problem to a one-dimensional integral.
\begin{align*}
& \frac{1}{(\epsilon/4)^d} \int_{H_0(R/4)^c} g(\|\theta\|) d\theta \\ &\leq \frac{1}{(\epsilon/4)^d} \int_{B(0, R/4)^c} g(\|\theta\|) d\theta
\\ & = \frac{1}{(\epsilon/4)^d} \int_{R/4}^\infty g(r) S_r dr
\\ &\leq \frac{2\pi^{d/2}}{(\epsilon/4)^d\Gamma(d/2)} \int_{R/4}^\infty g(r) r^{d-1} dr,
\end{align*}
where $S_r$ is the ``surface area" of any ball in $\R^d$ with radius $r$, which is $2\pi^{d/2}r^{d-1}/\Gamma(d/2)$.
\end{proof}

A more manageable quantity for the right-hand-side of Lemma~\ref{lem:tail-summation} can be derived.
\begin{lemma}\label{lem:tail-summation-with-stirling}
For any $\epsilon, d > 0$,
\begin{align*}
\frac{2\pi^{d/2}}{(\epsilon/4)^d\Gamma(d/2)} &\leq \left(\frac{20}{\epsilon \sqrt{d}}\right)^d.
\end{align*}
\end{lemma}

\begin{proof}
\cite[Theorem 1]{jameson2015} provides a Stirling lower bound for the gamma function:
\begin{align*}
\Gamma(d/2) \geq \frac{\sqrt{2\pi} (d/2)^{d/2}}{e^{d/2} \sqrt{d/2}}.
\end{align*}
We also upper bound $\sqrt{d}$ by $(2.9/2)^{d/2}$. The overall bound of $(20/\epsilon\sqrt{d})^d$ comes from rounding numbers up to the nearest integer.
\end{proof}

Next, we apply Lemma~\ref{lem:tail-summation-with-stirling} to power decay functions.
\begin{lemma}\label{lem:power-function-summation}
Let $\Te$ be an $\epsilon$-discretization of $\R^d$. For any $R \geq 3\epsilon$ and $q > d$,
\begin{align*}
\sum_{\theta \in \Te \cap B(\ts, R)^c} \frac{1}{\| \theta - \ts \|^q} &\leq \left(\frac{20}{\epsilon \sqrt{d}} \right)^d \frac{(4/R)^{q-d}}{q-d}.
\end{align*}
\end{lemma}

\begin{proof}
This is a straight-forward application of Lemmas~\ref{lem:tail-summation} and~\ref{lem:tail-summation-with-stirling}.
\begin{align*}
\sum_{\theta \in \Te \cap B(\ts, R)^c} \frac{1}{\| \theta - \ts \|^q} &\leq \left(\frac{20}{\epsilon \sqrt{d}} \right)^d \int_{R/4}^\infty \frac{r^{d-1}}{r^q} dr\\
 &= \left(\frac{20}{\epsilon \sqrt{d}} \right)^d \left[ \frac{r^{d-q}}{d-q} \right]_{R/4}^\infty\\
 &= \left(\frac{20}{\epsilon \sqrt{d}} \right)^d \frac{(R/4)^{d-q}}{q-d}.
\end{align*}
\end{proof}

In our applications, the decaying functions will often be taken to the $\alpha n$ power for some $\alpha \in [0, 1]$. For power decay in that case, Lemma~\ref{lem:power-function-summation} can be used to derive a bound that is exponential in dimension and is stable if $\epsilon$ is proportional to $1/\sqrt{n}$.

\begin{lemma}\label{lem:power-decay-without-epsilon}
Assume $n \geq (d+1)/\alpha b$, $a > 0$, and $R \geq 4a^{1/b} \vee 3\epsilon$. Then
\begin{align*}
& \sum_{\theta \in \Te \cap B(\ts, R)^c} \left( \frac{a}{\| \theta - \ts \|^b} \right)^{\alpha n} \\ &\leq \left( \frac{4R}{\epsilon \sqrt{n \alpha b \log(R/4a^{1/b})}} \right)^d.
\end{align*}
\end{lemma}

\begin{proof}
Start with Lemma~\ref{lem:power-function-summation}, then apply the assumption that $b\alpha n - d \geq 1$.
\begin{align*}
& \sum_{\theta \in \Te \cap B(\ts, R)^c} \left( \frac{a}{\| \theta - \ts \|^b} \right)^{\alpha n} \\ &= a^{\alpha n} \sum_{\theta \in \Te \cap B(\ts, R)^c} \frac{1}{\| \theta - \ts \|^{b\alpha n}}\\
 &\leq a^{\alpha n} \left( \frac{20}{\epsilon\sqrt{d}} \right)^d \frac{(4/R)^{b\alpha n - d}}{b\alpha n - d}\\
 &\leq \left( \frac{20 R}{4\epsilon\sqrt{d}} \right)^d \left(\frac{4a^{1/b}}{R}\right)^{b\alpha n}\\
 &= \left( \frac{20 R}{4\epsilon \sqrt{nd}} \right)^d n^{d/2} \left(\frac{4a^{1/b}}{R}\right)^{b\alpha n}.
\end{align*}
Assuming $4a^{1/b} < R$, the quantity $n^{d/2} (\frac{4a^{1/b}}{R})^{b\alpha n}$ is maximized at $n = \tfrac{d}{2b \alpha\log(R/4a^{1/b})}$. Substituting this critical value and rounding up gives us the desired bound.
\end{proof}

Suppose the sample size is not large enough for Lemma~\ref{lem:power-decay-without-epsilon} to be valid. If the summand is multiplied by a Gaussian-shaped function, then it is still possible to derive a bound that is stable if $\epsilon$ is proportional to $1/\sqrt{n}$, although the dependence on dimension becomes worse. Lemmas~\ref{lem:power-decay-reversed-epsilon-condition} and~\ref{lem:power-decay-middle-epsilon-condition} splits the problem into two additional ranges for $n$.

\begin{lemma}\label{lem:power-decay-reversed-epsilon-condition}
Assume $n \leq (d-1)/\alpha b$, $a, \kappa > 0$, and $R \geq 4a^{1/b} \vee 3\epsilon$. Then
\begin{align*}
& \sum_{\theta \in \Te \cap B(\ts, R)^c} e^{-\kappa \|\theta\|^2} \left( \frac{a}{\| \theta - \ts \|^b} \right)^{\alpha n} \\ &\leq 2 e^{\kappa \|\ts\|^2} \left( \frac{4 \sqrt{2 \pi e} \sqrt{d \vee a^{2/b} \kappa}}{\epsilon \sqrt{n \alpha b \kappa}} \right)^d.
\end{align*}
\end{lemma}

\begin{proof}
By Lemma~\ref{lem:squared-norm-bound}, we can upper bound $\| \theta \|^2$ in terms of $\|\theta - \ts\|^2$ and $\|\ts\|^2$.
\begin{align*}
e^{-\kappa \| \theta \|^2} &\leq e^{-\frac{\kappa}{2} \| \theta - \ts \|^2 + \kappa \|\ts\|^2}.
\end{align*}
Using this, we apply Lemma~\ref{lem:tail-summation}.
\begin{align*}
& \sum_{\theta \in \Te \cap B(\ts, R)^c} e^{-\kappa \|\theta\|^2} \left( \frac{a}{\| \theta - \ts \|^b} \right)^{\alpha n} \\ &\leq e^{\kappa \|\ts\|^2} a^{\alpha n} \sum_{\theta \in \Te \cap B(\ts, R)^c} e^{-\frac{\kappa}{2} \|\theta - \ts\|^2} \left( \frac{1}{\| \theta - \ts \|^b} \right)^{\alpha n}\\
 &\leq e^{\kappa \|\ts\|^2} a^{\alpha n} \frac{2\pi^{d/2}}{(\epsilon/4)^d\Gamma(d/2)} \int_{R/4}^\infty e^{-\kappa r^2/2} r^{d-b\alpha n-1} dr.
\end{align*}

The integral from $R/4$ to $\infty$ can be upper bounded by the integral from $0$ to $\infty$. Then we change the variable to $r^2$ and compare the integrand to a Gamma distribution's density.
\begin{align*}
& \int_{R/4}^\infty e^{-\kappa r^2/2} r^{d-b\alpha n-1} dr \\ &\leq \int_{R/4}^\infty e^{-\kappa r^2/2} r^{d-b\alpha n-1} dr\\
 &= \frac{1}{2} \int_0^\infty e^{-\kappa r^2/2} (r^2)^{(d-b\alpha n-2)/2} (2 r dr)\\
 &= \frac{\Gamma(\tfrac{d-b\alpha n}{2})}{2(\kappa/2)^{(d-b\alpha n)/2}}.
\end{align*}

Applying the Stirling upper and lower bounds for the gamma function \cite[Theorem 1]{jameson2015}, we arrive at the bound
\begin{align*}
& \sum_{\theta \in \Te \cap B(\ts, R)^c} e^{-\kappa \|\theta\|^2} \left( \frac{a}{\| \theta - \ts \|^b} \right)^{\alpha n} \\ &\leq e^{\kappa \|\ts\|^2} a^{\alpha n} \frac{2\pi^{d/2}}{(\epsilon/4)^d\Gamma(d/2)} \frac{\Gamma(\tfrac{d-b\alpha n}{2})}{2(\kappa/2)^{(d-b\alpha n)/2}}\\
 &\leq e^{\kappa \|\ts\|^2} \left( \frac{4 \sqrt{2 \pi}}{\epsilon \sqrt{n \kappa}} \right)^d n^{d/2} \left( \frac{a^{2/b} \kappa e}{d}\right)^{b \alpha n/2} \times \\ & \qquad e^{1/6(d - b\alpha n)} \sqrt{1 - \tfrac{b \alpha n}{d}}^{d - b\alpha n - 1}.
\end{align*}

We have assumed that $d - b\alpha n \geq 1$, so we can upper bound the final two factors by $e^{1/6}$ and $1$. If $d \leq a^{2/b} \kappa e$, then we substitute $d/b\alpha$ for $n$ to upper bound $n^{d/2} (\frac{a^{2/b} \kappa e}{d})^{b \alpha n/2}$ by $(\tfrac{a^{2/b} \kappa e}{b\alpha})^{d/2}$. Otherwise, the optimizer of the product of these two factors is $n = \frac{d}{b\alpha \log(d/a^{2/b}\kappa e) }$. Substituting this quantity into the second factor and $d/b\alpha$ into the first, we bound the product of the two factors by $(\tfrac{d}{b\alpha e})^{d/2}$. Using $e^{1/6} \leq 2$ and $d/e \leq ed$ simplifies the bound.
\end{proof}

\begin{lemma}\label{lem:power-decay-middle-epsilon-condition}
Assume $n \in (\frac{d-1}{\alpha b}, \frac{d+1}{\alpha b})$, $a, \kappa > 0$, and $R \geq 4a^{1/b} \vee 3\epsilon$. Then
\begin{align*}
& \sum_{\theta \in \Te \cap B(\ts, R)^c} e^{-\kappa \|\theta\|^2} \left( \frac{a}{\| \theta - \ts \|^b} \right)^{\alpha n} \\ &\leq e^{\kappa \|\ts\|^2} \left(\frac{20}{\epsilon \sqrt{n \alpha b}}\right)^d \left( \frac{22}{R^3} + 2\sqrt{\kappa} \right).
\end{align*}
\end{lemma}

\begin{proof}
Begin as in Lemma~\ref{lem:power-decay-reversed-epsilon-condition}. Our assumption about $n$ ensures that the exponent of $r$ is negative in the integral. First assume $R/4 < 1$.
\begin{align*}
& \int_{R/4}^\infty e^{-\kappa r^2/2} r^{d-b\alpha n-1} dr \\ &= \int_{R/4}^1 e^{-\kappa r^2/2} r^{d-b\alpha n-1} dr + \int_1^\infty e^{-\kappa r^2/2} r^{d-b\alpha n-1} dr\\
 &\leq \int_{R/4}^1 r^{d-b\alpha n-1} dr + \int_1^\infty e^{-\kappa r^2/2} dr\\
 &\leq \int_{R/4}^1 r^{-2} dr + \int_0^\infty e^{-\kappa r^2/2} dr\\
 &\leq \int_{R/4}^\infty r^{-2} dr + \frac{1}{2} (\sqrt{2 \pi \kappa})\\
 &\leq 22/R^3 + 2\sqrt{\kappa}.
\end{align*}
If $R/4 \geq 1$, then the integral is bounded by $\int_{R/4}^\infty e^{-\kappa r^2/2} dr$ which remains less than our bound.

We use Lemma~\ref{lem:tail-summation-with-stirling} for the coefficient of the integral.
\begin{align*}
& \sum_{\theta \in \Te \cap B(\ts, R)^c} e^{-\kappa \|\theta\|^2} \left( \frac{a}{\| \theta - \ts \|^b} \right)^{\alpha n} \\ &\leq e^{\kappa \|\ts\|^2} a^{\alpha n} \left(\frac{20}{\epsilon \sqrt{d}}\right)^d \left( \frac{22}{R^3} + 2\sqrt{\kappa} \right)\\
 &= e^{\kappa \|\ts\|^2} n^{d/2} a^{\alpha n} \left(\frac{20}{\epsilon \sqrt{nd}}\right)^d \left( \frac{22}{R^3} + 2\sqrt{\kappa}\right).
\end{align*}
If $a < 1$, the product $n^{d/2} a^{\alpha n}$ is maximized at $n = -d/2a\log\alpha$. Substituting that into the second factor gives $e^{-d/2}$; we bound the first factor using the fact that $n \leq (d+1)/b\alpha \leq 2d/b\alpha$. If $a \geq 1$, substitute $2d/b\alpha$ in both factors. In either case, the product is bounded by $(\tfrac{\sqrt{d}[1 \vee a^{2/b}]}{\sqrt{b\alpha}})^d$.
\end{proof}

\subsection{Hellinger affinity}

\begin{lemma}\label{lem:hellinger-affinity-cauchy-schwarz}
Let $P$ and $Q$ be probability measures on $(\X, \A)$. For any event $\h \in \A$,
\begin{align*}
A(P, Q) &\leq \sqrt{P \h} \sqrt{Q \h} + \sqrt{P \h^c} \sqrt{Q \h^c}.
\end{align*}
\end{lemma}

\begin{proof}
Let $p$ and $q$ be densities of $P$ and $Q$ with respect to a common dominating measure $\mu$. We use the Cauchy-Schwarz inequality.\footnote{Andrew R. Barron pointed out  to the authors that this Lemma is simply an application of the Data Processing Inequality when the random variable is processed by the indicator function of $\h$, and it is useful in hypothesis testing theory.}
\begin{align*}
A(P, Q) &= \int_{\x} \sqrt{p(x) q(x)} d\mu(x)\\
 &= \int_{\h} \sqrt{p(x) q(x)} dx + \int_{\h^c} \sqrt{p(x) q(x)} d\mu(x)\\
 &\leq \sqrt{\int_{\h} p(x) d\mu(x)} \sqrt{\int_{\h} q(x) d\mu(x)} + \\ & \qquad \sqrt{\int_{\h^c} p(x) d\mu(x)} \sqrt{\int_{\h^c} q(x) d\mu(x)}.
\end{align*}
\end{proof}

\begin{lemma}\label{lem:hellinger-affinity-first-moments}
Let $P$ and $Q$ be probability distributions on $\R^d$. If they both have finite first moments, then
\begin{align*}
A(P, Q) &\leq \frac{2( \E \| X - \E X\| + \E \| Y - \E Y \|)}{\| \E X - \E Y \|},
\end{align*}
where $X \sim P$ and $Y \sim Q$.
\end{lemma}

\begin{proof}
Let $\h$ denote the halfspace containing $\E X$ and demarcated by the perpendicular bisector of the path from $\E X$ to $\E Y$.

The $P$-probability of the complement of $\H$ is bounded by the probability of the complement of a ball within $\H$. We use Markov's inequality to bound the probability that the deviation $\| X - \E X \|$ is larger than $\| \E X - \E Y \|/2$.
\begin{align*}
P \h^c &\leq P \, B(\E X, \tfrac{\| \E X - \E Y \|}{2})^c\\ 
 &\leq \frac{\E \| X - \E X \|}{\| \E X - \E Y \|/2}.
\end{align*}

The same logic also allows us to bound $Q \h$. Now invoke Lemma~\ref{lem:hellinger-affinity-cauchy-schwarz}.
\begin{align*}
A(P, Q) &\leq \sqrt{P \h} \sqrt{Q \h} + \sqrt{P \h^c} \sqrt{Q \h^c}\\
 &\leq \sqrt{Q \h} + \sqrt{P \h^c}\\
 &\leq \frac{\E \| Y - \E Y \|}{\| \E X - \E Y \|/2} + \frac{\E \| X - \E X \|}{\| \E X - \E Y \|/2}.
\end{align*}
\end{proof}

\begin{lemma}\label{lem:hellinger-affinity-bound-median}
Let $P$ and $Q$ be probability distributions on $\R$. Suppse they have medians $m_P$ and $m_Q$, that is, $P (-\infty, m_P] = 1/2$ and $Q (-\infty, m_Q] = 1/2$. Then
\begin{align*}
A(P, Q) \leq e^{-z^2/2},
\end{align*}
where $z$ is the $Q$ probability of the interval defined by an open endpoint at $m_P$ and a closed endpoint at $m_Q$.
\end{lemma}

\begin{proof}
Assume without loss of generality that $m_P \leq m_Q$. Apply Lemma~\ref{lem:hellinger-affinity-cauchy-schwarz} using the interval $(-\infty, m_P]$ for $\h$.
\begin{align*}
A(P, Q) &\leq \sqrt{P \h} \sqrt{Q \h} + \sqrt{P \h^c} \sqrt{Q \h^c}\\
 &= \frac{1}{\sqrt{2}} (\sqrt{Q \h} + \sqrt{Q \h^c})\\
 &= \frac{1}{\sqrt{2}} (\sqrt{\tfrac{1}{2} - z} + \sqrt{\tfrac{1}{2} + z})\\
 &= \frac{1}{\sqrt{2}} \sqrt{1 + \sqrt{1 - 4z^2}}\\
 &\leq 1 - z^2/2\\
 &\leq e^{-z^2/2}.
\end{align*}
We recommend using mathematical software to algebraically verify the second-to-last step.
\end{proof}

Regarding Lemma~\ref{lem:hellinger-affinity-bound-median}, note that $A$ is symmetric in its arguments, so letting $z$ be the $P$ probability of the interval provides a valid bound as well.

\begin{lemma}\label{lem:hellinger-affinity-bound-marginal-median}
Let $P$ and $Q$ be probability distributions on $\R^d$. Assume they have marginal median vectors $m_P = (m_P^{(1)}, \ldots, m_P^{(d)})$ and $m_Q = (m_Q^{(1)}, \ldots, m_Q^{(d)})$, and assume that $Q$ has marginal densities $q_1, \ldots, q_d$ with respect to Lebesgue measure. Let $R \geq 0$. Then for all $Q$ with $m_Q \in B(m_P, R)$,
\begin{align*}
A(P, Q) \leq e^{-c\| m_Q - m_P \|^2},
\end{align*}
where
\begin{align*}
c := \tfrac{1}{2d} \min_{j \in \{1, \ldots, d\}} \, \min_{x \in [m_Q^{(j)} - R, m_Q^{(j)} + R]} q_j(x)^2.
\end{align*}
\end{lemma}

\begin{proof}
Let $P^*$ and $Q^*$ be the marginal distributions along the coordinate with the largest absolute difference between $m_P$ and $m_Q$; call the coordinates in this direction $m_P^*$ and $m_Q^*$. By Lemma~\ref{lem:hellinger-affinity-bound-median},
\begin{align*}
A(P^*, Q^*) &\leq e^{-z^2/2},
\end{align*}
where $z$ is the $Q^*$ probability of the interval from  $m_P^*$ to $m_Q^*$. It is at least as large as the absolute difference between the coordinates times the minimum value of the density $q^*$ in the interval between them. The largest squared coordinate difference is at least as large as the average squared coordinate difference, that is
\begin{align*}
|m_Q^* - m_P^*|^2 \geq \tfrac{1}{d} \|m_Q - m_P\|^2.
\end{align*}
The $1/2$ factor in $c$ comes from Lemma~\ref{lem:hellinger-affinity-bound-median}.

The marginal distributions $P^*$ and $Q^*$ can be produced by ``processing" draws from $P$ and $Q$. Because Hellinger affinity is a monotonically decreasing transformation of squared Hellinger divergence, the Data Processing Inequality implies that $A(P, Q) \leq A(P^*, Q^*)$.
\end{proof}

\textbf{Lemma~\ref{lem:exponential-affinity-guassian-bound}}
\begin{proof}
Let $r$ denote the family's carrier function and $\psi$ denote the log-partition function.
\begin{align}\label{eq:exponential-affinity}
A(\Pts, \Pt) &:= \int_\X \sqrt{\pt^*(x) \pt(x)} dx\nonumber\\
 &= \int_\X r(x) e^{\frac{1}{2}(\ts + \theta)'\phi(x) - \frac{1}{2}(\psi(\ts) + \psi(\theta))} dx\nonumber\\
 &= e^{-\frac{1}{2}(\psi(\ts) + \psi(\theta))} \int_\X r(x) e^{\frac{1}{2}(\ts + \theta)'\phi(x)} dx\nonumber\\
 &= e^{-\frac{1}{2}(\psi(\ts) + \psi(\theta))} e^{\psi((\ts + \theta)/2)}\nonumber\\
 &= e^{-[(\psi(\ts) + \psi(\theta))/2 - \psi((\ts + \theta)/2)]}.
\end{align}
The exponent is a negative Jensen difference with a distribution that puts $1/2$ mass on each of $\theta$ and $\ts$. Its expectation is $(\ts + \theta)/2$, and its variance is $\| \theta - \ts \|^2/4$. Applying Lemma~\ref{lem:jensen-difference-bounds},
\begin{align*}
& \frac{\psi(\ts) + \psi(\theta)}{2} - \psi\left(\frac{\ts + \theta}{2}\right) \\ & \geq \frac{\| \theta - \ts \|^2}{8} \inf_{\tt \in \Theta} \lambda_d(\nabla \nabla' \psi(\tt)).
\end{align*}
It is a well-known fact about exponential families that $\nabla \nabla' \psi(\tt)$ is equal to the covariance matrix of the sufficient statistic vector $\C_{X \sim P_{\tt}} \phi(X)$.

\end{proof}

\subsection{Entropy of subprobability measures}

We extend the notion of entropy to more general measures. Let $Q$ be a measure on a countable set $\x$. Then we define its entropy
\begin{align*}
H(Q) := \sum_{x \in \x} q(x) \log \frac{1}{q(x)},
\end{align*}
where $q(x) := Q(\{x\})$ is the density of $Q$ with respect to counting measure.\footnote{One can likewise extend the notion of differential entropy $h$ for Borel measures on $\R^d$ by using Lebesgue measure rather than counting measure. Lemmas~\ref{lem:entropy-extension} and~\ref{lem:entropy-extension-inequality} also hold for differential entropy when $Q$ is a finite Borel measure on $\x = \R^d$.}

\begin{lemma}\label{lem:entropy-extension}
Let $Q$ be a finite measure on a countable set $\x$, and let $\widetilde{Q} := \tfrac{1}{Q \x} Q$ be the normalized version of $Q$. Then
\begin{align*}
H(Q) = (Q \x) H(\widetilde{Q}) + (Q \x) \log \tfrac{1}{Q \x}.
\end{align*}
\end{lemma}

\begin{proof}
\begin{align*}
H(Q) &:= \sum_{x \in \x} q(x) \log \frac{1}{q(x)}\\
 &= (Q \x) \sum_{x \in \x} \frac{q(x)}{Q \x} \log \frac{1 / Q \x}{q(x)/Q \x}\\
 &= (Q \x) \sum_{x \in \x} \frac{q(x)}{Q \x} \log \frac{1}{q(x)/Q \x} + \\ & \qquad (Q \x) \sum_{x \in \x} \frac{q(x)}{Q \x} \log \frac{1}{Q \x}.
\end{align*}
\end{proof}

\begin{lemma}\label{lem:entropy-extension-inequality}
Let $Q$ be a subprobability measure on a countable set $\x$, and let $\widetilde{Q} := \tfrac{1}{Q \x} Q$ be the normalized version of $Q$. Then
\begin{align*}
H(Q) \leq H(\widetilde{Q}) + 1/e.
\end{align*}
\end{lemma}

\begin{proof}
We apply Lemma~\ref{lem:entropy-extension}, noting that $Q \x \leq 1$ and that the function $-z \log z$ has maximum $1/e$.
\end{proof}

In particular, for any subprobability distribution $Q$, $H(Q) \leq \log |\x| + 1/e$. A cleaner inequality holds if $|\x| \geq 3$.

\begin{lemma}\label{lem:entropy-extension-size}
Let $Q$ be a subprobability measure on a countable set $\x$. If $|\x| \geq 3$, then
\begin{align*}
H(Q) \leq \log |\x|.
\end{align*}
\end{lemma}

\begin{proof}
Consider the expression in Lemma~\ref{lem:entropy-extension}, first bounding $H(\widetilde{Q})$ by $\log |\x|$. The function $z \mapsto z \log |\x| - z \log z$ (with domain $[0, 1]$) is maximized at $e^{\log |\x|-1} \wedge 1$. When $e^{\log |\x|-1} \geq 1$, then the function is bounded by $z\log |\x|$ which is no greater than $\log |\x|$. This case applies when $|\x| \geq e$. Otherwise, the function's maximum value is $|\x|/e$. Thus, the proposed inequality does not hold for sets of size $1$ or $2$.
\end{proof}

\textbf{Proof of Lemma~\ref{lem:estimator-entropy-bound}}
\begin{proof}
Let $q$ denote the pmf of $\th$.
\begin{align*}
H(\th) &= \sum_{\theta \in \Te} q(\theta) \log \frac{1}{q(\theta)}.
\end{align*}
We will bound two parts of the summation separately: outside a ball centered at $\ts$ and then inside that ball.

The function $z \mapsto z \log (1/z)$ increases as $z$ goes from $0$ to $1/e$. If $\| \theta - \ts\| \geq 1/\sqrt{c}$, then the $\theta$ term of the entropy summation can only be increased by substituting the exponential probability bound
\begin{align*}
q(\theta) \log \frac{1}{q(\theta)} \leq e^{-c \| \theta - \ts\|^2} c \| \theta - \ts\|^2.
\end{align*}
Thus, outside the ball $B := B(\ts, R \vee c^{-1/2} \vee 3\epsilon)$, we can bound the summation by an integral using Lemma~\ref{lem:tail-summation} then compare the integral with a Gamma pdf.
\begin{align*}
& \sum_{\theta \in \Te \cap B^c} q(\theta) \log \frac{1}{q(\theta)} \\ &\leq \sum_{\theta \in \Te \cap B^c} e^{-c \| \theta - \ts\|^2} c \| \theta - \ts\|^2\\
 &\leq \frac{2 \pi^{d/2}}{(\epsilon/4)^d \Gamma(d/2)} \int_{[R \vee c^{-1/2} \vee 3\epsilon]/4}^\infty e^{cr^2} cr^2 r^{d-1} dr\\
 &\leq \frac{c \pi^{d/2}}{(\epsilon/4)^d \Gamma(d/2)} \int_0^\infty e^{cr^2} (r^2)^{d/2} 2r dr\\
 &= \frac{c \pi^{d/2}}{(\epsilon/4)^d \Gamma(d/2)} \frac{\Gamma(d/2 + 1)}{c^{d/2+1}}\\
 &= \frac{d}{2} \left( \frac{4\sqrt{\pi}}{\epsilon \sqrt{c}} \right)^d.
\end{align*}

Next, we need to bound the terms coming from grid-points inside $B$. $Q$ restricted to this subset of grid-points can be considered a subprobability measure. Because the ball has radius at least $3 \epsilon$, it contains enough grid-points for Lemma~\ref{lem:entropy-extension-size} to apply; thus the entropy of this subprobability is bounded by the log-cardinality of the ball's grid-points. The number of grid-points in the hypercube circumscribing the ball is no more than $(1+\tfrac{2[R \vee c^{-1/2} \vee 3\epsilon]}{\epsilon})^d$.
\begin{align*}
\sum_{\theta \in \Te \cap B} q(\theta) \log \frac{1}{q(\theta)} &\leq \log \left(1+\frac{2[R \vee c^{-1/2} \vee 3\epsilon]}{\epsilon}\right)^d.
\end{align*}
\end{proof}

\vskip 0.2in
\bibliographystyle{IEEEtran}
\bibliography{bibliography}

\begin{thebibliography}{10}
\providecommand{\url}[1]{#1}
\csname url@samestyle\endcsname
\providecommand{\newblock}{\relax}
\providecommand{\bibinfo}[2]{#2}
\providecommand{\BIBentrySTDinterwordspacing}{\spaceskip=0pt\relax}
\providecommand{\BIBentryALTinterwordstretchfactor}{4}
\providecommand{\BIBentryALTinterwordspacing}{\spaceskip=\fontdimen2\font plus
\BIBentryALTinterwordstretchfactor\fontdimen3\font minus
  \fontdimen4\font\relax}
\providecommand{\BIBforeignlanguage}[2]{{%
\expandafter\ifx\csname l@#1\endcsname\relax
\typeout{** WARNING: IEEEtran.bst: No hyphenation pattern has been}%
\typeout{** loaded for the language `#1'. Using the pattern for}%
\typeout{** the default language instead.}%
\else
\language=\csname l@#1\endcsname
\fi
#2}}
\providecommand{\BIBdecl}{\relax}
\BIBdecl

\bibitem{rissanen1978}
J.~Rissanen, ``Modeling by shortest data description,'' \emph{Automatica},
  vol.~14, no.~5, pp. 465--471, 1978.

\bibitem{barron1991}
A.~R. Barron and T.~M. Cover, ``Minimum complexity density estimation,''
  \emph{IEEE transactions on information theory}, vol.~37, no.~4, pp.
  1034--1054, 1991.

\bibitem{grunwald2007b}
P.~Gr{\"u}nwald and J.~Langford, ``Suboptimal behavior of bayes and mdl in
  classification under misspecification,'' \emph{Machine Learning}, vol.~66,
  no. 2-3, pp. 119--149, 2007.

\bibitem{renyi1961}
A.~R\'{e}nyi, ``On measures of entropy and information,'' in \emph{Fourth
  Berkeley Symposium on Mathematical Statistics and Probability}, vol.~1, 1961,
  pp. 547--561.

\bibitem{bhattacharyya1943}
A.~Bhattachayya, ``On a measure of divergence between two statistical
  population defined by their population distributions,'' \emph{Bulletin
  Calcutta Mathematical Society}, vol.~35, pp. 99--109, 1943.

\bibitem{erven2014}
T.~Van~Erven and P.~Harremo\"es, ``R\'{e}nyi divergence and
  {K}ullback-{L}eibler divergence,'' \emph{Information Theory, IEEE
  Transactions on}, vol.~60, no.~7, pp. 3797--3820, 2014.

\bibitem{barron2008}
A.~R. Barron, C.~Huang, J.~Q. Li, and X.~Luo, ``{The MDL Principle, Penalized
  Likelihoods, and Statistical Risk},'' \emph{Feschrift for Jorma Rissanen},
  2008.

\bibitem{chatterjee2014b}
S.~Chatterjee and A.~Barron, ``Information theoretic validity of penalized
  likelihood,'' in \emph{2014 IEEE International Symposium on Information
  Theory}.\hskip 1em plus 0.5em minus 0.4em\relax IEEE, 2014, pp. 3027--3031.

\bibitem{luo2009}
X.~Luo, ``Penalized likelihoods: fast algorithms and risk bounds,'' Ph.D.
  dissertation, Yale University, 2009.

\bibitem{chatterjee2014}
S.~Chatterjee, ``{Adaptation in Estimation and Annealing},'' Ph.D.
  dissertation, Yale University, 2014.

\bibitem{rissanen1986}
J.~Rissanen, ``Stochastic complexity and modeling,'' \emph{The Annals of
  Statistics}, pp. 1080--1100, 1986.

\bibitem{barron1998}
A.~R. Barron and N.~Hengartner, ``Information theory and superefficiency,''
  \emph{The Annals of Statistics}, vol.~26, no.~5, pp. 1800--1825, 1998.

\bibitem{zhang2006}
T.~Zhang, ``From $\epsilon$-entropy to {KL}-entropy: Analysis of minimum
  information complexity density estimation,'' \emph{The Annals of Statistics},
  vol.~34, no.~5, pp. 2180--2210, 2006.

\bibitem{li1999}
J.~Q. Li, ``{Estimation of Mixture Models},'' Ph.D. dissertation, Yale
  University, 1999.

\bibitem{brinda2018}
W.~D. Brinda, ``{Adaptive Estimation with Gaussian Radial Basis Mixtures},''
  Ph.D. dissertation, Yale University, 2018.

\bibitem{grunwald2007}
P.~D. Gr\"{u}nwald, \emph{The Minimum Description Length Principle}.\hskip 1em
  plus 0.5em minus 0.4em\relax MIT press, 2007.

\bibitem{jameson2015}
G.~Jameson, ``A simple proof of {S}tirling's formula for the gamma function,''
  \emph{The Mathematical Gazette}, vol.~99, no. 544, p.~68, 2015.

\end{thebibliography}



\end{document}